\documentclass[aop]{imsart}

\RequirePackage{amsthm,amsmath,amsfonts,amssymb,enumitem}
\RequirePackage[numbers]{natbib}
\RequirePackage[colorlinks,citecolor=blue,urlcolor=blue]{hyperref}
\RequirePackage{graphicx}

\startlocaldefs

\usepackage{constants}

\setlist[enumerate]{label=(\alph*)}
\numberwithin{equation}{section}

\theoremstyle{plain}
\newtheorem{theorem}{Theorem}[section]
\newtheorem{corollary}[theorem]{Corollary}
\newtheorem{claim}[theorem]{Claim}
\newtheorem{proposition}[theorem]{Proposition}
\newtheorem{lemma}[theorem]{Lemma}

\theoremstyle{remark}

\newtheorem{remark}[theorem]{Remark}
\newtheorem{assumption}{Assumption}

\newcommand{\EE}{\mathtt{E}}
\newcommand{\E}{\mathbb{E}}
\newcommand{\N}{\mathbb{N}}
\newcommand{\PPP}{\mathtt{P}}
\newcommand{\PP}{\mathbb{P}}
\newcommand{\R}{\mathbb{R}}

\newcommand{\bbone}{\boldsymbol 1}
\newcommand{\ei}{\mathtt{ei}}
\newcommand{\es}{\mathtt{es}}
\newcommand{\vbar}{\overline{v}}

\newcommand*\D{\mathop{}\!\mathrm{d}} 

\DeclareMathOperator*{\essinf}{ess\,inf}
\DeclareMathOperator*{\esssup}{ess\,sup}

\DeclareMathOperator{\var}{Var}
\DeclareMathOperator{\diam}{diam}

\newcommand{\cerny}[1]{\v Cern\'y} 

\newcommand{\parenthezises}[1]{\arabic{#1}}
\newconstantfamily{small}{
	symbol=c,
	format=\parenthezises,
}

\endlocaldefs

\begin{document}

\begin{frontmatter}

\title{On the tightness of the maximum of branching Brownian
  motion in random environment}

\runtitle{Tightness of the maximum of BBM in random environment}

\begin{aug}
  \author[A]{\fnms{Jiří}~\snm{Černý}
    \ead[label=e1]{jiri.cerny@unibas.ch}\orcid{0000-0001-7723-9245}},
  \author[B]{\fnms{Alexander}~\snm{Drewitz}
    \ead[label=e2]{adrewitz@uni-koeln.de}\orcid{0000-0002-5546-3614}}
  \and
  \author[A]{\fnms{Pascal}~\snm{Oswald}
    \ead[label=e3]{pascalamadeus.oswald@unibas.ch}}

  \address[A]{Department of mathematics and computer science,
    University of Basel\printead[presep={,\ }]{e1,e3}}

  \address[B]{
    Department Mathematik/Informatik, Universität zu Köln\printead[presep={,\ }]{e2}}
\end{aug}

\begin{abstract}
  We consider one-dimensional branching Brownian motion in spatially
  random branching environment (BBMRE) and show that for almost every
  realisation of the environment, the distributions of the maximal
  particle of the BBMRE re-centred around its median are tight as time
  evolves. This result is in stark contrast to the fact that the
  transition fronts in the solution to the randomised
  Fisher--Kolmogorov--Petrovskii--Piskunov (F-KPP) equation are, in
  general, not bounded uniformly in time. In particular, this highlights
  that---when compared to the settings of homogeneous branching Brownian
  motion and the F-KPP equation in a homogeneous environment---the
  introduction of a random environment leads to a much more intricate
  behaviour.
\end{abstract}



\end{frontmatter}

\section{Introduction}
\label{sec:introduction}

The behaviour of the position of the maximally---or, equivalently,
minimally---displaced particle in various variants of branching random
walk (BRW) and branching Brownian motion (BBM) has been the subject of
intensive research over the last couple of decades
\cite{Bramson1978maximaldisplacement, Bramson1983FKPP,
  Bramson2007tightnessBRW, AddarioBerry2009minima, HuShi2009minimal,
  Ai-13}. While initially most of the work focused on branching systems
with homogeneous branching rates, there has recently been an increased
activity in the investigation of branching random walks with
non-homogeneous branching rates that depend on either time or space
mostly in special \emph{deterministic} ways, see \cite{LaSe-88, LaSe-89,
  FangZeitouni2012BRWRE,FaZe-12b, BeBrHaHaRo-12, MaZe-13, Ma-13,
  BoHa-14,BoHa-15,CD2020, Kriechbaum2021subsequential,
  HoReSo-22,Kriechbaum2022BRWRE}.

In this article, we continue the study of the maximally displaced
particle in the model of branching Brownian motion with spatially random
branching environment (BBMRE). The study of BBMRE was initiated in
\cite{DS2022} as a tool for investigating properties of the randomised
Fisher--Kolmogorov--Petrovskii--Piskunov (F-KPP) equation, building
on the previous work \cite{CD2020} on a discrete-space analogue, i.e.\
branching random walk in i.i.d.~random environment (BRWRE). The
techniques developed in \cite{CD2020, DS2022} also lend themselves to
obtain refined information on the front of the solution of the randomised
F-KPP equation \cite{CDS2021}. Subsequently, the techniques and results
of \cite{CD2020} have been extended to the continuum space setting of
BBMRE in \cite{HoReSo-22}.

We complement the above body of research by addressing a seemingly
simple, but subtle problem that arises naturally, and which has also been
formulated as an open question in \cite{CD2020}. More precisely, we show
that the distributions of the position of the maximally displaced
particle of the BBMRE, when re-centred around its median, form a tight
family of distributions as time evolves. While establishing tightness
might a priori not look like an overly intricate problem, we take the
opportunity to emphasise that such a preconception is erroneous, see also
\cite{Bramson2007tightnessBRW,BramsonZeitouni2009recursion}. Our result
is particularly interesting as it sharply contrasts the result
established in \cite{CDS2021}, where it is shown that the transition
fronts of the solution to the randomised F-KPP equation are, in general,
unbounded in time. In the homogeneous setting, such a dichotomy cannot be
observed since, a fortiori, there is a duality between these two objects
in that tightness of the re-centred maximum of BBM is equivalent to the
uniform boundedness in time of the transition fronts of the solution to
the F-KPP equation.

\subsection{Homogeneous BBM and F-KPP equation}

To explain this duality in more detail, we start by recalling the model
in the homogeneous situation, which will also serve as a point of
reference throughout the article. For a (binary) branching Brownian
motion with homogeneous branching rate equal to one, started from a
single particle located at the origin at time $0$, we denote its maximal
displacement at time $t \ge 0$ by $M(t)$, and write
\begin{equation} \label{eqn:FKPPdistMax}
    w(t,x) = P(M(t)\ge x),
\end{equation}
for the probability that this displacement exceeds $x\in \mathbb R$.
Then, the function $w(t,x)$
solves a non-linear PDE, known as the
Fisher--Kolmogorov--Petrovskii--Piskunov equation,
\begin{equation}\label{eqn:hom_FKPP}
  \partial_t w(t,x)
  = \frac{1}{2}\partial_x^2w(t,x) + w(t,x)(1-w(t,x)),
  \qquad t>0, x\in \R,
\end{equation}
with the initial datum $w(0,\cdot) = \bbone_{(-\infty,0]}$ of Heaviside
type, see \cite{Ikeda1968branching,mckean1975KPP}. Moreover, it is well
known that as $t \to \infty$, the solution to \eqref{eqn:hom_FKPP}
approaches a \emph{travelling wave} $g$ in the following sense: for an
appropriate function $m : [0,\infty) \to [0,\infty)$ one has that
\begin{equation}
  \label{eqn:homogeneous_travelling_wave}
  w(t,m(t)+ \cdot)  \to   g \quad \text{uniformly as $t \to \infty$}
\end{equation}
for a decreasing function $g$ satisfying $\lim_{x\to\infty}g(x) = 0$ and
$\lim_{x\to-\infty}g(x) = 1$. A critical ingredient in the proof of this
convergence is that, again for $m(t)$ being chosen appropriately, one has
\begin{align}
  \label{eqn:homogeneous_stretching}
  \begin{split}
    &w(t, x+ m(t))\text{ is decreasing in $t$ for $x<0$, and}\\
    &w(t, x+ m(t))\text{ is increasing in $t$ for $x>0$}.
  \end{split}
\end{align}
Property \eqref{eqn:homogeneous_travelling_wave}
immediately yields for every $\varepsilon>0$ the existence of some
$r_{\varepsilon} \in (0,\infty)$ such that
\begin{equation}
  \label{eqn:homogeneous_F-KPP_tightness}
  w(t,m(t)+r_{\varepsilon})-w(t, m(t) -r_{\varepsilon})
  > 1-\varepsilon  \qquad
  \text{for all $t\ge 0$.}
\end{equation}
Put differently, the family $(M(t)-m(t))_{t\ge 0}$ is tight. Another,
essentially trivial, consequence of
\eqref{eqn:homogeneous_travelling_wave} is the uniform boundedness of the
width of the transition front of the solution to \eqref{eqn:hom_FKPP};
that is, that for every $\varepsilon \in (0,1/2)$,
\begin{equation}
  \label{eqn:homboundedness}
  \limsup_{t\to \infty} \diam \big( \{x \in \R : w(t,x)
      \in [\varepsilon,1-\varepsilon]\}\big) < \infty.
\end{equation}

In this context, it is worth pointing out that the above line of
reasoning implicitly uses the reflection symmetry of Brownian motion and
the homogeneity of the branching environment. As a consequence, this
proof technique breaks down in the presence of an inhomogeneous
environment, and the relationship between the solutions of the F-KPP
equation and the maximum of BBMRE becomes more intricate than that given
in \eqref{eqn:FKPPdistMax} and \eqref{eqn:hom_FKPP},
cf.~Section~\ref{ssec:PAM}.

\subsection{Randomised F-KPP equation}
In the inhomogeneous setting of a random potential, as considered in the
current paper, the respective randomised F-KPP equation has been
investigated in \cite{CDS2021}. In that source it has been established
that, for a canonical choice of random potentials $\xi$, the transition
front of the solution to the inhomogeneous F-KPP equation (which is
  discussed in more detail in Section~\ref{ssec:PAM})
\begin{equation}
  \label{eqn:ranFKPPintro}
  \partial_t w^{\xi}(t,x)
  = \frac{1}{2}\partial_x^2w^{\xi}(t,x)
  + \xi(x)w^{\xi}(t,x)(1-w^{\xi}(t,x)),
  \qquad t>0,x\in \R,
\end{equation}
with the initial condition $w^{\xi}(0,\cdot) = \bbone_{(-\infty,0]}$ does
not need to be uniformly bounded in time. More precisely, in contrast
to~\eqref{eqn:homboundedness}, it follows from
\cite[Theorem~2.3]{CDS2021} that there are random potentials $\xi$ within
the class of inhomogeneities considered in the current paper, such that
$\PP$-a.s.,~for all ${\varepsilon\in (0,1/2)}$,
\begin{equation}
\label{eqn:unbounded_TF}
\limsup_{t\to \infty} \diam \big( \{x \in \R : w^{\xi}(t,x)
    \in [\varepsilon,1-\varepsilon]\}\big) = +\infty.
\end{equation}

It might hence be surprising and is non-trivial to prove that for BBMRE
in the random potential $\xi$ we obtain tightness for the re-centred
family of maxima, and a novel approach is required in order to address
this situation adequately.

It is worthwhile to note that the PDE results of \cite{CDS2021} have been
obtained by taking advantage of almost exclusively probabilistic
techniques. In the current article, however, the probabilistic main
result will be proven via a combination of analytic and probabilistic
techniques.

\section{Definition of the model and the main result}
\label{sec:model_and_results}

We work with a model of \emph{branching Brownian motion in random
  branching environment} (BBMRE) introduced in \cite{CDS2021,DS2022} as a
continuous space version of the branching random walk in random
environment model studied in \cite{CD2020}. The random environment is
given by a stochastic process $\xi = (\xi(x) )_{x \in \R}$ defined on
some probability space $(\Omega,\mathcal{F},\PP)$ which fulfils the
following assumptions.

\begin{assumption}
  \phantomsection
  \label{ass:environment}
  \begin{itemize}
    \item $\xi$ is stationary, that is, for every $h \in \R$ one has
    \begin{equation}
      \label{eqn:stationarity_condition}
      (\xi(x))_{x \in \R} \stackrel{(d)}{=} (\xi(x+h) )_{x\in \R}.
    \end{equation}

    \item
    $\xi$ fulfils a  \emph{$\psi$-mixing} condition:  There exists a
    continuous non-increasing function
    $\psi:[0,\infty)\to[0,\infty)$ satisfying
    $\sum_{k=1}^\infty \psi(k)<\infty$ such that (using the notation
    $\mathcal F_A = \sigma (\xi (x):x\in A)$ for
    $A\subset \mathbb R$) for all
    $Y\in L^1(\Omega ,\mathcal F_{(-\infty,j]},\mathbb P)$, and all
    $Z\in L^1(\Omega ,\mathcal F_{[k,\infty)},\mathbb P)$ we have
    \begin{equation}
      \begin{split}
        \label{eqn:MIX}
        \big| \E\big[Y-\E[Y] \mid  \mathcal F_{[k,\infty)}\big]\big|
          &\leq \E[|Y|]\, \psi(k-j),
          \\
          \big| \E\big[Z-\E[Z] \mid  \mathcal F_{(-\infty,j]}\big]\big|
          &\leq \E[|Z|]\,\psi({k-j}).
      \end{split}
    \end{equation}
    (Note that these conditions imply the ergodicity of $\xi$ with
    respect to the usual shift operator.)

    \item The sample paths of $\xi$ are $\PP$-a.s.~locally Hölder
    continuous, that is, for almost every $\xi$ there exists
    $\alpha = \alpha(\xi)\in (0,1)$ and for every compact $K\subseteq \R$
    a constant $C = C(K,\xi)>0$ such that
    \begin{equation}
      \label{eqn:Hoelder_condition}
      |\xi(x)-\xi(y)| \le C |x-y|^{\alpha}, \qquad \text{for all } x,y\in K.
    \end{equation}

    \item $\xi$ is uniformly bounded in the sense that the essential
    infimum and supremum of the random variable $\xi(0)$ (and thus also
      of $\xi(x)$, $x\in \mathbb R$, by \eqref{eqn:stationarity_condition})
    satisfy
    \begin{equation}
      \label{eqn:uniformly_bounded}
      0<\ei := \essinf \xi(0)<\esssup \xi(0) =: \es < \infty.
    \end{equation}

  \end{itemize}
\end{assumption}

In the current article, we do not explicitly make use of the mixing
condition. However, in particular in
Section~\ref{sec:Perturbation_results_for_the_PAM}, we will employ some
of the results developed in \cite{CDS2021,DS2022} which depend on this
mixing assumption.

The dynamics of BBMRE started at a position $x \in \R$ is as follows.
Given a realisation of the environment $\xi$,  we place one particle at
$x$ at time $t = 0$. As time evolves, the particle follows the trajectory
of a standard Brownian motion $(X_t)_{t\ge 0}$. Additionally and
independently of everything else, while at position $y$, the particle is
killed with rate $\xi(y)$. Immediately after its death, the particle is
replaced by $k$ independent copies at the site of death, according to
some fixed offspring distribution $(p_k)_{k \ge 1}$. All $k$ descendants
evolve independently of each other according to the same stochastic
diffusion-branching dynamics.

We denote by $\PPP_x^{\xi}$ the quenched law of a BBMRE, started at $x$
and write $\EE_x^{\xi}$ for the corresponding expectation. Moreover, we
denote by $N(t)$ the set of particles alive at time $t$. For any particle
$\nu \in N(t)$ we denote by $(X_s^{\nu})_{s \in [0,t]}$ the spatial
trajectory of the genealogy of ancestral particles of $\nu$ (unique at
  any given time) up to time $t$. Our main interest lies in the
maximally displaced particle of the BBMRE at time $t$,
\begin{equation*}
  M(t) := \sup\{X_t^{\nu} : \nu \in N(t)\}.
\end{equation*}

Throughout this article, we deal with supercritical branching such that
the offspring distribution has second moments and particles always have
at least one offspring.

\begin{assumption}
  \label{ass:OD}
  The offspring distribution $(p_k)_{k\ge 1}$ satisfies
  \begin{equation}
    \label{ass:offspring_distribution}
    \sum_{k = 1}^{\infty}kp_k =: \mu>1,
    \qquad\text{and}\qquad
    \sum_{k = 1}^{\infty}k^2p_k =:\mu_2<\infty.
  \end{equation}
\end{assumption}

Under these assumptions, the maximally displaced particle $M(t)$ satisfies
a law of large numbers for some non-random asymptotic velocity
$v_0 \in (0,\infty)$; that is, for $\mathbb P$-a.e.~$\xi$ one has
\begin{equation}
   \label{eqn:MtLLN}
    \lim_{t\to \infty} \frac{M(t)}t = v_0,  \quad
    \text{$\PPP^\xi_0$-a.s.},
\end{equation}
see \cite[Corollary 1.5]{HoReSo-22}. (Note also that convergence in
  probability follows from classical results of Freidlin, \cite[Theorem
    7.6.1]{Freidlin1985functionalintegration}.) The asymptotic velocity
can be characterised as the unique positive root of the \emph{Lyapunov
  exponent} $\lambda$, which is a deterministic function
$\lambda : \R \to \R$ that admits the representation
\begin{equation}
  \label{eqn:Lyapunov_exponent}
  \lambda(v) = \lim_{t \to \infty} \frac{1}{t}
  \ln \EE_{0}^{\xi}\big[\big|\{ \nu \in N(t) : X_t^{\nu} \ge
      vt\}\big|\big],
  \quad \PP\text{-a.s.}
\end{equation}
Under Assumptions \ref{ass:environment} and \ref{ass:OD}, the function
$\lambda$ is non-increasing on $[0,\infty)$, concave,  and there exists a
critical value $v_c \ge 0$ such that $\lambda$ is linear on $[0,v_c]$ and
strictly concave on $[v_c,\infty)$, see
e.g.~\cite[Corollary~3.10]{DS2022}. As in \cite{CD2020,CDS2021,DS2022} we
make the following technical assumption.

\begin{assumption}
  \label{ass:VEL}
  We only consider BBMREs whose asymptotic speed satisfies
  \begin{equation}
    \label{eqn:VEL}
    v_0>v_c.
  \end{equation}
\end{assumption}

Essentially, this condition allows for the introduction of a \emph{tilted}
probability measure in the \emph{ballistic phase}, under which a
Brownian particle $(X_t)_{t\ge0}$ moves \emph{on average} with speed $v_0$
up to time $t$, cf.~Section~\ref{sec:tilted_measure}. By the same
argument as in \cite[Lemma~A.4]{CD2020} one can show that \eqref{eqn:VEL}
is satisfied by a rich class of environments. We refer also to
\cite[Section 4.4]{DS2022} for a more in-depth discussion on the
condition \eqref{eqn:VEL} and in particular to
\cite[Proposition~4.10]{DS2022} where environments are constructed which
satisfy Assumption~\ref{ass:environment}, but for which \eqref{eqn:VEL}
fails. Due to the length of the construction we do not replicate it here.

Finally,  we also define for $\varepsilon \in (0,1)$ the quenched
quantiles for the distribution of $M(t)$ where the process is started at
the origin,
\begin{equation}
  \label{eqn:meps}
  m_{\varepsilon}^{\xi}(t) :=
  \inf \big\{ y \in \R : \PPP_0^{\xi}( M(t)\le y)\ge \varepsilon\big\}.
\end{equation}
For notational convenience, we omit the subscript when $\varepsilon =1/2$
and write $m^{\xi}(t)$ for the median of the distribution.

With the above notation at our disposal, we can state our main result.

\begin{theorem}
  \label{thm:BBMRE_is_tight}
  Under Assumptions \ref{ass:environment}--\ref{ass:VEL},
  for almost every realisation of the environment $\xi$, the
  family $\big(M(t)-m^{\xi}(t) \big)_{t\ge0}$ is tight under $\PPP_0^{\xi}$.
\end{theorem}

\begin{remark}
  Note that Theorem \ref{thm:BBMRE_is_tight} also remains valid if for
  any $\varepsilon \in (0,1)$, the quantity $m^\xi(t)$ is replaced by
  $m_\varepsilon^\xi(t)$.
\end{remark}

This result should be contrasted with the behaviour
\eqref{eqn:unbounded_TF} of transition fronts of solutions to the
inhomogeneous F-KPP equation \eqref{eqn:ranFKPPintro} discussed in the
introduction. More precisely, in \cite[Theorem~2.3, Theorem~2.4]{CDS2021}
environments $\xi$ satisfying
Assumptions~\ref{ass:environment}--\ref{ass:VEL} of the present paper
were constructed for which the width of the transition front grows
logarithmically in time, along a sub-sequence. That is, for small enough
$\varepsilon>0$, there exist times and positions
$(t_n)_n , (x_n)_n \in \Theta(n)$, and a function
$\varphi  \in \Theta( \ln n)$ such that
\begin{equation} \label{eqn:nonMon}
  w^{\xi}(t_n,x_n) \ge w^{\xi}(t_n,x_n+\varphi(n))+\varepsilon.
\end{equation}
(Note, that this not only implies \eqref{eqn:unbounded_TF}, but also the
  spatial non-monotonicity of the functions $w^{\xi}(t,\cdot)$.) The
existence of environments for which Theorem~\ref{thm:BBMRE_is_tight}, and
\eqref{eqn:unbounded_TF} or \eqref{eqn:nonMon} hold simultaneously seems
unintuitive, as it sharply contrasts with the homogeneous case. Indeed,
in the latter, as indicated by \eqref{eqn:homogeneous_stretching} and
\eqref{eqn:homogeneous_F-KPP_tightness}, the standard reasoning for
deducing  the tightness of BBM is by the uniform boundedness in time of
transition fronts for the corresponding homogeneous F-KPP solutions. We
will explain the reason for this apparent discrepancy later in the paper
(see Section~\ref{ssec:strategy}).

Questions of tightness also arise naturally and have been addressed in
many other classes of models. In \cite{BramsonZeitouni2009recursion},
 analytic tools have been  developed in order to establish tightness
for a class of discrete time models whose distribution functions satisfy
certain recursive equations, analogous to the F-KPP equation in the case
of BBM. These tools are powerful and were applied and  adapted to establish
tightness for several models, e.g.\ \cite{AddarioBerry2009minima,
  Bolthausen2011recursions, Dembo2021limitlaw, FangZeitouni2012BRWRE,
  HuShi2009minimal, Neuman2021maximal} to name a few.
For BBM in a periodic environment, \cite{Lubetzky2018MaximumOB} used an
analytic result on the F-KPP front in periodic environment
\cite{Hamel2016periodic} which directly implies tightness.

In the context of the discrete space model of \cite{CD2020},
sub-sequential tightness along a deterministic sequence is shown for the
quenched and annealed law of the  maximally displaced particle in
\cite{Kriechbaum2021subsequential} using a Dekking-Host type argument.
Our method relies crucially on analytic properties of solutions to the
F-KPP equation, and differs from the approaches in the above mentioned
articles.

The tightness result of Theorem~\ref{thm:BBMRE_is_tight} naturally
suggests the question whether the random variables $M(t)-m^\xi(t)$
converge in distribution as $t\to\infty$. Supported by the numerical
simulations presented in Figure~\ref{fig:nonconvergence}, we conjecture
that the answer to this question is negative.

\begin{figure}[ht]
  \centering \includegraphics[width=7cm]{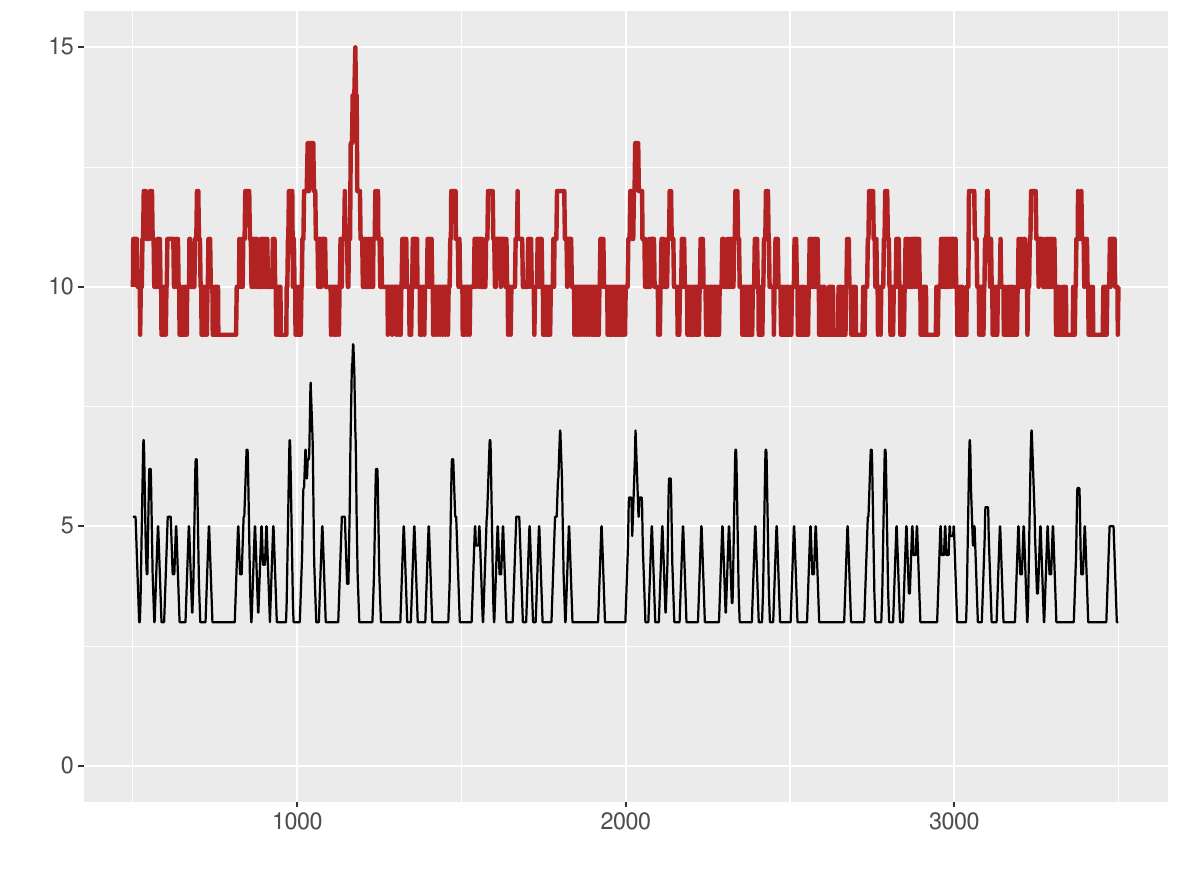}
  \includegraphics[width=7cm]{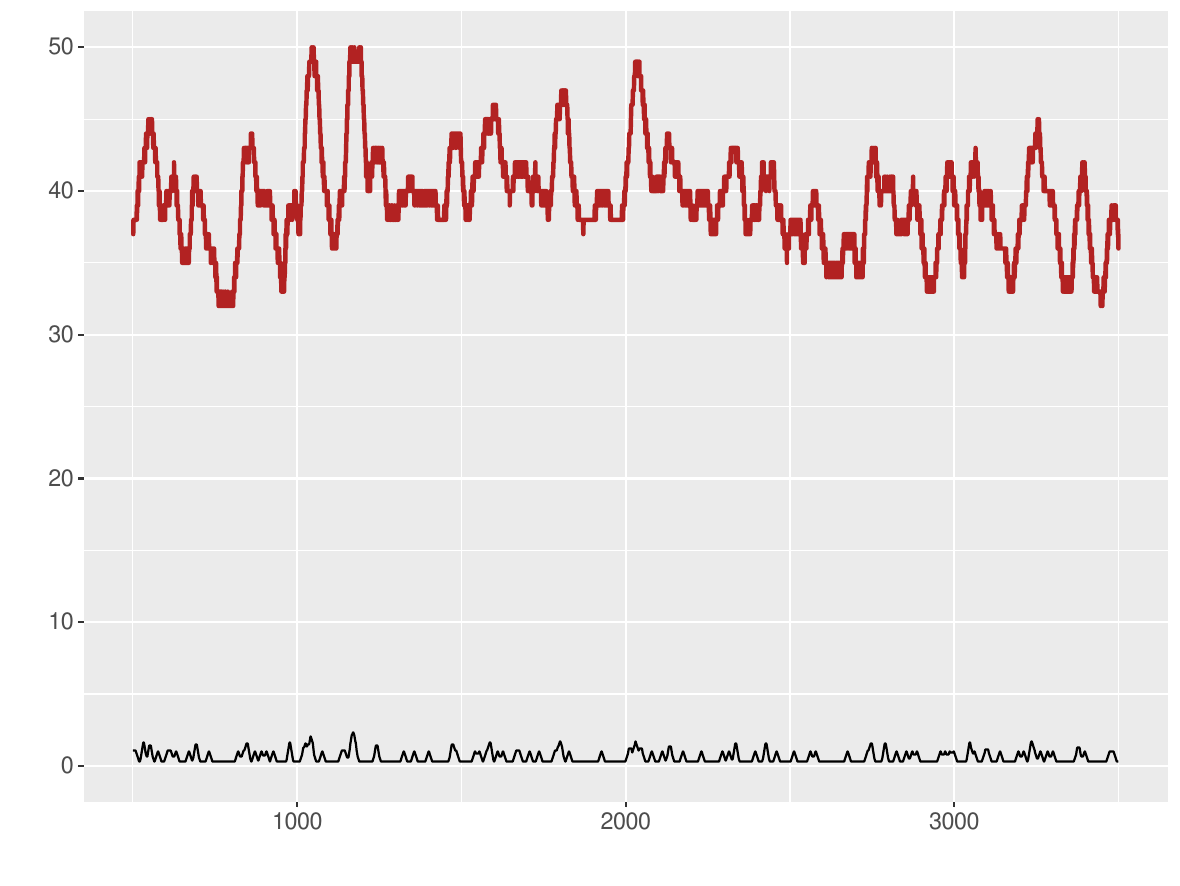}
  \caption{Numerical simulations suggesting that the distributions of
    $M(t)-m^\xi(t)$ do not converge as $t\to\infty$; the red line shows
    the dependence of the ``spread'' of this distribution, that is of
    $m^\xi_{0.99}(t)-m^\xi_{0.01}(t)$, on the median $m^\xi(t)$. The
    black line shows the corresponding potential $\xi(x)$ as a function
    of $x$. The simulations were performed for a discrete-space model,
    for realisations of $\xi$ from  two different distributions (left and
      right panel). In both cases, note the similarity of the red and
    black line, in the sense that at times $t$ when the median $m^\xi(t)$
    reaches an area where $\xi$ is large, the spread of $M(t)$ tends to
    be large as well.}
    \label{fig:nonconvergence}
\end{figure}

\begin{remark}
  As observed above, in  \cite{HoReSo-22}   the authors prove an annealed
  functional central limit theorem for the position of the maximally
  displaced particle $M(t)$ of BBMRE in the setting described here. In
  our notation, this means that for some $\sigma^2 \in (0,\infty)$, the
  sequence of processes
  \begin{equation*}
    [0,\infty) \ni t
    \mapsto \frac{M(nt)-v_0 nt}{\sqrt{\sigma^2n}}, \quad n \in \N,
  \end{equation*}
  under $\PP \times \PPP_x^\xi$ converges weakly in $C([0,\infty))$ to
  standard Brownian motion.

  Using McKean's representation (see Proposition \ref{pro:McKean} below),
  reflecting the potential around the origin by defining
  $\widetilde \xi(-y):= \xi(y)$ for all $y \in \R$, and using its
  stationarity, we obtain  that for $x \in \R$  the solution
  $w^{\widetilde \xi}(t,x)$ to \eqref{eqn:ranFKPPintro} with initial
  condition $\bbone_{(-\infty,0]}$ has the same $\PP$-distribution as
  $\PPP_{0}^{ \xi}(M(t) \ge x)$. As a consequence, the (functional)
  central limit theorem \cite[Corollary 1.6]{DS2022} for the front of
  this solution $w^{\widetilde \xi}(t,x)$ at level $\varepsilon \in (0,1)$
  (note that  $\widetilde \xi$  still satisfies the conditions imposed on
    $\xi$ for \cite[Corollary 1.6]{DS2022} to hold) entails a
  (non-functional) central limit theorem for $m_\varepsilon^\xi(t)$ as
  defined in \eqref{eqn:meps} as well; that is, the sequence of random
  variables $(m^\xi(nt)-v_0nt)/{\sqrt{\sigma^2n}}$, $n \in \N$, converges
  weakly to a $\mathcal N(0, \sigma^2)$-distributed random variable. In
  combination with Theorem \ref{thm:BBMRE_is_tight}, we recover a
  non-functional form of the above central limit theorem for $M(t)$ as
  well. That is, the sequence of random variables
  $(M(nt)-v_0nt)/{\sqrt{\sigma^2n}}$, $n \in \N$, converges weakly to a
  $\mathcal N(0, \sigma^2)$-distributed random variable under
  $\PP \times \PPP_x^\xi$.
\end{remark}

\subsection{Strategy of the proof}
\label{ssec:strategy}

We now explain the main ideas behind the proof of
Theorem~\ref{thm:BBMRE_is_tight}, and also comment on the seeming
discrepancy between this theorem and properties \eqref{eqn:unbounded_TF},
\eqref{eqn:nonMon} of the solutions to the randomized F-KPP equation.

The first ingredient of the proof is the well-known duality
between the distribution of $M(t)$ and the solutions to the randomised
F-KPP equation \eqref{eqn:ranFKPPintro}. In the spatially non-homogeneous
case this duality states (see Section~\ref{ssec:PAM} below)
\begin{equation}
  \label{eqn:dualityintro}
  w^y(t,x) = \PPP_x^{\xi}(M(t) \ge y),
\end{equation}
where $w^y$ is the solution to \eqref{eqn:ranFKPPintro} with the initial
condition $w^y(0,\cdot) = \bbone_{[y,\infty)}$. In order to prove the
tightness of $(M(t)-m^\xi(t))_{t\ge 0}$  it suffices to check that the
difference of the $\varepsilon$-quantile and $(1-\varepsilon)$-quantile
of $M(t)$ is bounded uniformly in time. More precisely, for every
$\varepsilon >0$ we need to find $\Delta = \Delta(\varepsilon )<\infty$ so that
for all $t>0$ and $x_t = x_t(\varepsilon )\in \mathbb R$ characterized via
\begin{equation}
  \label{eqn:tightintroaa}
  \varepsilon  = \PPP_0^\xi(M(t) \ge x_t) =  w^{x_t}(t,0)
\end{equation}
it holds that
\begin{equation}
  \label{eqn:tightintrobb}
  1-\varepsilon
  < \PPP_0^\xi\big(M(t) \ge x_t-\Delta\big)
  =  w^{x_t-\Delta}(t,0).
\end{equation}

We note here that this already provides an indication that the
above mentioned discrepancy is only apparent: While
properties~\eqref{eqn:unbounded_TF} and \eqref{eqn:nonMon} are linked to
the dependency of $w^y(t,x)$ on the spatial variable $x$, the tightness
of $M(t)$ is linked to its dependency on the initial
condition~$\bbone_{[y,\infty)}$.

To show that \eqref{eqn:tightintrobb} holds true, we first exploit the fact
that solutions to \eqref{eqn:ranFKPPintro} increase quickly to~$1$,
once they move away from~$0$. In connection with \eqref{eqn:tightintroaa}
this fact implies (cf. Corollary~\ref{cor:almost_surely_finite_time} below)
that for some $T = T(\varepsilon )<\infty$, uniformly in $\xi$ and $t$
large,  we  have
\begin{equation}
  \label{eqn:tightintroca}
  1-\varepsilon< w^{x_t}(t+T,0) =: \overline w(t,0).
\end{equation}
Note that $\overline w$ is the solution of \eqref{eqn:ranFKPPintro} with
the initial condition $w^{x_t}(T, \cdot)$. In view of this,
\eqref{eqn:tightintrobb} follows, if we can show that for some $\Delta$
sufficiently large, we have uniformly in $t$ large that
\begin{equation}
  \label{eqn:tightintrocc}
  w^{x_t-\Delta}(t,0)>\overline w(t,0).
\end{equation}

Proving inequality \eqref{eqn:tightintrocc} directly at the spatial
coordinate $x=0$ seems to be difficult. It requires comparing two
solutions to the randomised F-KPP equation \eqref{eqn:ranFKPPintro} in
the regime where they are away from~$0$ and~$1$, and where various
approximations to them, e.g. by linearisation, are not precise enough. To
work around this difficulty, we take advantage of the \emph{Sturmian
  principle} for solutions of parabolic PDEs which we recall in
Section~\ref{ssec:sturmian_principle}. As we will see in the proof of
Theorem~\ref{thm:BBMRE_is_tight}, this principle implies that
\eqref{eqn:tightintrocc} follows if, for some $v>0$, we can show the
modified inequality
\begin{equation}
  \label{eqn:tightintrodd}
  w^{x_t-\Delta}(t,-vt)>\overline w(t,-vt).
\end{equation}
The advantage of inequality \eqref{eqn:tightintrodd} is that if $v$ is
sufficiently away from 0, then both its sides are very close to $0$ and
thus can be controlled using linearisation techniques or the first order
Feynman-Kac formulas. The proof of \eqref{eqn:tightintrodd} is still
rather technical and is provided in the Lemma~\ref{lem:contradiction}
below.

Finally, we return to the discrepancy between the divergence of the width
of the transition front \eqref{eqn:unbounded_TF} and the tightness proved
in Theorem~\ref{thm:BBMRE_is_tight}. The proof of
\eqref{eqn:unbounded_TF} in~\cite{CDS2021} exploits the fact that it is
easy to construct potentials where, for $n$ large, in any interval $[0,n]$
there is a subinterval of length $\Theta(\log n)$ where the potential is
close to $\ei$, closely followed by a subinterval of the same length
where the potential almost equals $\es$. If $\es/\ei > 2$, then the
existence of such subintervals forces a creation of ``bumps'' in the
solutions to the randomised F-KPP equation, as illustrated in
Figure~\ref{fig:bump}. The creation of such bumps directly leads to
\eqref{eqn:unbounded_TF}. It turns out that the existence of those
subintervals does not make inequality \eqref{eqn:tightintrodd} invalid,
even, e.g., if $x_t$ is located in a subinterval where the potential is
close to $\es$ and $x_t-\Delta$ in a subinterval where $\xi$ is almost
$\ei$. Via the duality \eqref{eqn:dualityintro}, this is related to the
established intuition that the behaviour of the maximum of branching
processes is more easily influenced by the environment close to
their starting point than by that elsewhere.

\begin{figure}[ht]
  \centering
  \includegraphics{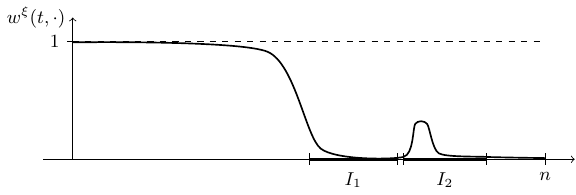}

  \caption{A bump in the solution to \eqref{eqn:ranFKPPintro} which develops
    shortly after the front of the solution (moving to the
      right) reaches the interval $I_1$
    where $\xi$ is close to $\ei$, and which is followed by the interval
    $I_2$
    with $\xi$ close to $\es$,  if $\es/\ei > 2$. Both intervals are of
    length $\Theta(\log n)$.
  }
  \label{fig:bump}
\end{figure}

\subsubsection*{Organisation of the article}

The rest of the paper is organised as follows. In Section
\ref{sec:preliminaries} we first recall some well-established facts, such
as the duality \eqref{eqn:dualityintro}, the Feynman-Kac formulas for the
solution of the randomised F-KPP equation and of its linearisation, the
parabolic Anderson model. We then discuss a first application of the
Sturmian principle to our setting. Section \ref{sec:tilted_measure}
reviews tilted measures, which, on a technical level, will play the role
of a suitable ``gauging-measure'' under which we can compare the terms in
\eqref{eqn:tightintrodd}. In
Section~\ref{sec:Perturbation_results_for_the_PAM} we explain how to
extend the arguments in \cite{CDS2021,DS2022} in order to obtain a
spatial and temporal perturbation result for solutions of the parabolic
Anderson model. This perturbation result is then applied in
Section~\ref{sec:proofof_contradiction_lemma} where we prove the key
technical lemma, which is related to inequality \eqref{eqn:tightintrodd}.
Finally, Section~\ref{sec:proofof_tightness} contains the proof of the
main theorem.

\subsubsection*{Notational conventions}

We use $c$, $C$, $c'$ etc.~to denote positive finite constants whose
value may change during computations. Indexed versions $c_1$, $c_2$, $C_1$,
etc.~are used to denote positive finite constants that are fixed
throughout the article. Their dependence on other parameters is made
explicit when the constants are introduced, with an additional convention
given in Section~\ref{sec:Perturbation_results_for_the_PAM}.

\section{Preliminaries}
\label{sec:preliminaries}

This section recalls two important and well known probabilistic tools
which feature heavily in the proof of our main theorem. Furthermore, we
make precise the Sturmian principle alluded to above.

\subsection{The randomised F-KPP equation and its linearisation}
\label{ssec:PAM}

As already mentioned in the introduction, there is a fundamental link
between branching Brownian motion  and solutions to the homogeneous F-KPP
equation. It is often attributed to McKean \cite{mckean1975KPP}, but can
already be found in Skorokhod \cite{Skorohod1964branching} and Ikeda,
Nagasawa and Watanabe \cite{Ikeda1968branching}. Such a connection can
also be extended to the setting of random branching rates, as we now
describe. For this purpose, assume we are given an offspring distribution
$(p_k)_{k\ge 1}$ as in \eqref{ass:offspring_distribution}. We then
consider the random semilinear heat equation
\begin{equation*}
  \label{eqn:FKPP}
  \begin{aligned}
    \partial_t w(t,x)&= \frac{1}{2}\partial_x^2 w(t,x) + \xi(x)F(w(t,x)),
    \qquad&&t>0,x\in \R,\\
    w(0,x) &=w_0(x), &&x\in \R,
  \end{aligned}
  \tag{F-KPP}
\end{equation*}
where the non-linearity $F: [0,1]\to [0,1]$ is given by
\begin{equation}
  \label{eqn:non-linearity}
  F(w) = (1-w) - \sum_{k=1}^{\infty}p_k(1-w)^k, \qquad w \in [0,1].
\end{equation}
The adaptation of McKean's representation of solutions to
\eqref{eqn:FKPP} takes the following form.

\begin{proposition}
  \label{pro:McKean}
  For any function $w_0 : \R \to [0,1]$ which is the pointwise limit of
  an increasing sequence of continuous functions, and for any bounded,
  locally Hölder continuous function $\xi : \R \to (0,\infty) $, there
  exists a solution to \eqref{eqn:FKPP}  which is continuous on
  $(0,\infty)\times \R$ and which, for $t \in [0,\infty)$ and $x \in \R$,
  can be represented as
  \begin{equation}
    \label{eqn:McKean}
    w(t,x)
    = 1- \EE_x^{\xi}\Big[ \prod_{\nu \in N(t)}\big(1-w_0(X_t^{\nu})\big)\Big].
  \end{equation}
\end{proposition}

A proof of this proposition can be found e.g.~in
\cite[Proposition~2.1]{DS2022}; the formulation in that source is under
slightly more restrictive conditions, but it transfers verbatim to the
assumptions we impose above.

A crucial consequence of Proposition \ref{pro:McKean} is that
the solution $w^y$ of \eqref{eqn:FKPP} with Heaviside-like initial
condition $w^y_0 = \bbone_{[y,\infty)}$, for $y \in \R$, is linked to the
distribution function of $M(t)$ via the identity
\begin{equation}
  \label{eqn:FKPP_solution}
  w^y(t,x) = \PPP_x^{\xi}(M(t) \ge y).
\end{equation}

\begin{remark}
  \label{rem:muis2}
  It is common practice in the F-KPP literature to normalise the
  non-linearity $F$ in such a way that its derivative at the origin is
  one. Using \eqref{ass:offspring_distribution} it is easy to
  check that in our case, $F'(0) = \mu-1$. In other words, the standard
  normalisation of equation \eqref{eqn:FKPP} corresponds to a branching
  processes for which the offspring distribution has mean $\mu = 2$, as
  is also assumed in \cite{DS2022}. In
  \eqref{ass:offspring_distribution},
  we assume only that $\mu>1$ and do not a priori work
  under the usual F-KPP normalisation. Nevertheless, given any such
  offspring distribution $(p_k)_{k\ge 1}$ with mean $\mu \neq 2$ and a
  corresponding BBMRE in environment $\xi$, one can always transform it
  into another BBMRE in a rescaled environment, so that the transformed
  process is in the usual normalisation and has the same distribution as
  the original process. Indeed, the transformation
  defined by
  \begin{equation*}
    \xi \to (\mu-1)\xi, \quad p_1 \to \frac{\mu+p_1-2}{\mu-1},
    \quad\text{and} \quad
    p_k \to \frac{p_k}{\mu-1} \text{ for } k\ge2,
  \end{equation*}
  yields a new offspring distribution with mean two. Moreover, rescaling
  the environment guarantees that \eqref{eqn:FKPP}, and the law
  $\PPP_x^{\xi}$ are invariant under this transformation. After
  rescaling, it holds that $F'(0) =1$ and $\mu_2>2$; hence, in light of
  this reasoning,  we will from now on always assume that
  \begin{equation}
    \label{eqn:ODprime}
     \mu = 2, \qquad F'(0)=1, \quad\text{and}\quad \mu_2>2.
  \end{equation}
 Observe also that by \eqref{eqn:non-linearity} this implies that
  \begin{equation}
    \label{eqn:Fderivatives}
    F'(w)\le 1, \quad\text{and}\quad F''(w)\ge -\mu_2+2 \qquad \text{for
      all } w\in [0,1].
  \end{equation}
\end{remark}

Another PDE related to BBMRE, which we make use of later on, is the
linearisation of \eqref{eqn:FKPP}, known as the \emph{parabolic Anderson
  model} (PAM),
\begin{equation*}
  \tag{PAM}
  \label{eqn:PAM}
  \begin{aligned}
    \partial_t u(t,x)&= \frac{1}{2}\partial_x^2 u(t,x) +\xi(x)u(t,x),
    \qquad &&t>0,x\in \mathbb R\\
    u(0,x) &= u_0(x), && x \in \R.
  \end{aligned}
\end{equation*}
The PAM has been the subject of intense investigation in its own right,
see e.g.~\cite{Koenig2016PAM} and reference therein for a comprehensive
overview; our main interest, however, lies in space and time perturbation
results that have been developed for its solution in
\cite{CDS2021,DS2022}. These will be considered in more detail in
Section~\ref{sec:Perturbation_results_for_the_PAM}.

An important strategy for probabilistically investigating the solutions
to the equations \eqref{eqn:FKPP} and \eqref{eqn:PAM} is via analysing
their Feynman-Kac representations. In what comes below we denote by $P_x$
the probability measure under which the process denoted by
$(X_t)_{t \ge0}$ is a standard Brownian motion started at $x\in \mathbb R$.
The corresponding expectation operator is denoted by $E_x$. We also make
repeated use of the abbreviation $E_x[f;A]$ for $E_x[f\bbone_A]$.

\begin{proposition}
  \label{pro:Feynman-Kac}
  Under Assumptions \ref{ass:environment} and
  \ref{ass:OD}, the unique non-negative solution $u$ of
  \eqref{eqn:PAM} is given by
  \begin{equation}
    \label{eqn:Feynman-Kac_PAM}
    u(t,x) = E_x\Big[\exp\Big\{\int_0^t \xi(X_r)\D r \Big\}\, u_0(X_t) \Big],
    \qquad t\ge 0, x\in \mathbb R,
  \end{equation}
  and the unique non-negative solution $w$ of  \eqref{eqn:FKPP} fulfils
  \begin{equation}
    \label{eqn:Feynman-Kac_FKPP}
    w(t,x) = E_x \Big[ \exp\Big\{\int_0^t\xi(X_r)
        \widetilde{F}(w(t-r,X_r))\D r\Big\} \, w_0(X_t)\Big],
    \qquad t\ge 0, x\in \mathbb R,
  \end{equation}
  where $\widetilde{F}(w) = F(w)/w$ for $w \in (0,1]$, which can be
  continuously extended to
  $\widetilde{F}(0) = \lim_{w \to 0+}\widetilde{F}(w)
  = \sup_{w \in (0,1]}\widetilde{F}(w) =1$.
\end{proposition}

See e.g.~\cite[(1.32), (1.33)]{Bramson1983FKPP} for references to the
former. Note that the Feynman-Kac representation
\eqref{eqn:Feynman-Kac_PAM} is explicit, while
\eqref{eqn:Feynman-Kac_FKPP} is not (in the sense that the expressions on
  both sides of the latter equation involve $w$).

Taking advantage of the above, a (direct) link between the PAM and BBMRE
can be derived by combining the Feynman-Kac representation
\eqref{eqn:Feynman-Kac_PAM} of the solution to \eqref{eqn:PAM} with a
\emph{many-to-one formula}, see e.g.~\cite[Proposition 2.3]{DS2022}, in
order to arrive at the representation
\begin{equation*}
  u(t,x) = \EE_x^{\xi}\Big[ \sum\limits_{\nu \in N(t)}u_0(X_t^{\nu})\Big]
\end{equation*}
of solutions to \eqref{eqn:PAM}.

\subsection{Sturmian principle}
\label{ssec:sturmian_principle}

In this section, we present the analytic ingredient of our proof of
Theorem~\ref{thm:BBMRE_is_tight}.  As explained in the introduction (see
  around~\eqref{eqn:tightintrocc}), we are interested in differences of
the type $W(\cdot,\cdot) = w^{y_1}(\cdot,\cdot)- w^{y_2}(\cdot+T,\cdot)$
for some $T>0$, and $y_2>y_1$, where we recall that for any $y\in \R$, we
denote by $w^y$ the solution of \eqref{eqn:FKPP} with initial condition
$w_0= \bbone_{[y,\infty)}$. It is immediate that  the function $W$
satisfies the linear parabolic equation
\begin{align}
  \label{eqn:linear_parabolic_equation}
  \begin{aligned}
    \partial_t W(t,x) &= \frac{1}{2} \partial_x^2W(t,x)+ G(t,x)W(t,x),
    \qquad  &&t>0,x\in \mathbb R,\\
    W(0,x) &=\bbone_{[y_1,\infty)}(x)-w^{y_2}(T,x), \quad  &&x \in \R,
  \end{aligned}
\end{align}
where $G$ is the bounded measurable function defined by (using the
  convention $F'(0)=1$, cf.~Remark~\ref{rem:muis2})
\begin{equation}
  \label{eqn:linear_parabolic_equation_G}
  G(t,x) =
  \begin{cases}
    \xi(x) \,
    \frac{F(w^{y_1}(t,x)) - F(w^{y_2}(t+T,x))}{w^{y_1}(t,x)-w^{y_2}(t+T,x)} ,
    & \text{if }w^{y_1}(t,x) \neq w^{y_2}(t+T,x),\\
    \xi(x) , &\text{if } w^{y_1}(t,x) = w^{y_2}(t+T,x).
  \end{cases}
\end{equation}

Let us state the following simple observation, which will be used at
various stages in the following: By Proposition \ref{pro:McKean} it
follows that
\begin{equation} \label{eqn:sol01bd}
  0< w^{y_2}(T,x)< 1 \quad  \text{ for all } T>0 \text{ and }x \in \R.
\end{equation}
As a consequence, the
initial condition of \eqref{eqn:linear_parabolic_equation} has
exactly one zero-crossing, and it is located at $y_1$.

In the analysis literature, it has been known for a long time that the
cardinality of the set of zero-crossings of solutions to linear parabolic
equations is monotonically non-increasing in time, with the earliest
reference dating back to at least an article by Charles Sturm in 1836,
cf.~\cite{Sturm1836memoires}. Nevertheless, despite this result having
been known for almost two centuries by now, it was not until the eighties
of the last century that Sturm's ideas really revived in the theory of
linear and non-linear parabolic equations, see,
e.g.,~\cite{Angenent1988zeroset, angenent1991nodalproperties,
  Ducrot2014propagatingterrace, evanswilliams1999zerocrossings,
  Nadin2012criticalwaves} for a non-exhaustive list. In this list, the
ideas in \cite{evanswilliams1999zerocrossings} stand out, as they involve
a simple and purely probabilistic proof by interpreting the linear
parabolic partial differential equations as generators of Markov
processes and reducing the study of the zero-crossings to the study of
Markovian transition operators acting on signed measure spaces. A more
complete history and a detailed discussion of the Sturmian principle and
its applications can be found in \cite{Galaktionov2004sturmian}.
In this context, it is interesting to note that already in their
seminal article on the F-KPP equation, Kolmogorov, Petrovskii and
Piskunov also make use of a Sturmian principle for equations of the form
\eqref{eqn:linear_parabolic_equation}, see
\cite[Theorem~11]{KPP1937etude}, which is proved using a parabolic
maximum principle.

We include a version of such a result which is formulated to fit our
purpose; a more general version of this result can be found in
\cite{Nadin2012criticalwaves}. Note that the assumptions in particular
fit the setting of a single zero-crossing in the initial datum.

\begin{lemma}[{{\cite[Proposition 7.1]{Nadin2012criticalwaves}}}]
  \label{lem:sturmian_principle}
  For any $t_0\in \R$, let $G  \in L^{\infty}((t_0,\infty)\times \R)$ and
  assume
  $W \in C((t_0,\infty)\times \R)\cap L^{\infty}((t_0,\infty)\times \R)$
  to be a weak solution of
  \begin{align*}
    \partial_t W(t,x) &= \frac{1}{2} \partial_x^2W(t,x)+ G(t,x)W(t,x),
    \quad  &&t>t_0,x\in \mathbb R,\\
    W(t_0,x) &= W_{t_0}(x),\quad &&x \in \R,
  \end{align*}
  where $W_{t_0}\not\equiv 0$ is piecewise continuous and bounded in $\R$,
  such that for some $z_{t_0} \in \R$ one has
  \begin{equation*}
    W_{t_0}(x)\le 0,\text{ if $x<z_{t_0}$,} \qquad \text{and} \qquad
    W_{t_0}(x)\ge 0,\text{ if $x>z_{t_0}$.}
  \end{equation*}
  Then, for all $t>t_0$ there exists a unique point
  $z(t) \in [-\infty,\infty]$ such that
  \begin{equation*}
    W(t,x)<0,\text{ if  $x<z(t)$,}
    \qquad\text{and}\qquad
    W(t,x)> 0,\text{ if $x>z(t)$}.
  \end{equation*}
\end{lemma}

As a first application of Lemma~\ref{lem:sturmian_principle}, let us
consider the behaviour of the solutions of \eqref{eqn:FKPP} when the
discontinuity of the Heaviside-type initial condition tends to infinity.
For this purpose, in order to obtain a non-trivial limit, we perform an
appropriate temporal shift. More precisely, we introduce for a given
realisation of the environment $\xi$, any $y \in \R$ and any
$\varepsilon>0$ the ``temporal quantile at the origin'' as
\begin{equation}
  \label{eqn:temporal_quantile}
  \tau_y^{\varepsilon} :=\inf\{ t \ge 0 :  w^y(t,0)\ge \varepsilon\}.
\end{equation}
Since $\PP$-a.s.~we have $\lim_{t\to\infty} w^y(t,0)=1$ (due to, e.g.,
  \cite[Theorem 7.6.1]{Freidlin1985functionalintegration}),
$\tau_y^\varepsilon$ is finite. By the continuity of $w^y$ on
$(0,\infty) \times \R$, cf.~Proposition~\ref{pro:McKean}, the quantity
$\tau_y^{\varepsilon}$ satisfies
${w^y(\tau_y^{\varepsilon},0) = \varepsilon}$. From
\eqref{eqn:FKPP_solution} it follows that
$y\mapsto w^y(t,0)$ is
decreasing, and thus $y\mapsto \tau_{y}^{\varepsilon}$ is
increasing. By the law of
large numbers for the maximal displacement (cf.\
  \eqref{eqn:MtLLN}), it follows
readily that $\lim\limits_{y\to\infty}\tau_y^{\varepsilon}= \infty$.

The shift by $\tau_y^{\varepsilon}$ allows to establish the following
result, which follows already from
\cite[Lemma~7.3]{Nadin2012criticalwaves}. Nevertheless, we provide its
short proof here for the sake of completeness and as an illustration of
how Lemma~\ref{lem:sturmian_principle} can be used in our context.

\begin{proposition}
  \label{pro:divergent_initial_value}
  For every $\varepsilon \in (0,1)$, $\PP$-a.s., the limit
  \begin{equation}
    \label{eqn:divergent_initial_value}
    w^{\infty}_\varepsilon(t,x)
    := \lim_{y \to \infty} w^y(\tau_y^{\varepsilon}+t,x)
  \end{equation}
  exists locally uniformly in $(t,x)\in \R^2$, and is a global-in-time
  (that is, for all $t\in \mathbb R$)
  solution to \eqref{eqn:FKPP}.
\end{proposition}

The limiting function $w^{\infty}_{\varepsilon}$ plays a role comparable
to that of a travelling wave solution of the homogeneous F-KPP equation,
cf.~\eqref{eqn:homogeneous_travelling_wave}. However, unlike in the
homogeneous situation outlined in the introduction,
$w^{\infty}_{\varepsilon}$ does not directly provide an argument for
tightness because we lack a suitable quantitative control of the random
variables $\tau_y^{\varepsilon}$ as $y$ varies. Nonetheless, the result
of Proposition \ref{pro:divergent_initial_value} plays a vital role in
our proof of tightness. We restrict ourselves to providing a proof of the
convergence for $t>0$ only, as this is sufficient for our purposes in
what follows.

\begin{proof}[Proof of Proposition~\ref{pro:divergent_initial_value}]
  Fix $y_1<y_2$ and for
  $t\ge-\tau_{y_1}^{\varepsilon}
  = -\tau_{y_1}^{\varepsilon} \vee -\tau_{y_2}^{\varepsilon}$
  (recall that the latter identity follows from the monotonicity of
    $y \mapsto \tau_y^\varepsilon$ observed below
    \eqref{eqn:temporal_quantile}) define the function
  $W(t,x) :=w^{y_1}(t+\tau_{y_1}^{\varepsilon},x)
  -w^{y_2}({t+\tau_{y_2}^{\varepsilon},x})$.
  Then, similarly as for \eqref{eqn:linear_parabolic_equation} and
  \eqref{eqn:linear_parabolic_equation_G}, it follows that
  \begin{equation}
    \label{eqn:cauchy_problem_difference_equation}
    \partial_t W(t,x) = \frac{1}{2}\partial_x^2W(t,x)+ G(t,x)W(t,x),
    \qquad t>-\tau_{y_1}^{\varepsilon}, x\in\R,
  \end{equation}
  where $G$ is given by
  \begin{equation*}
    G(t,x) =
    \begin{cases}
      \xi(x)\,
      \frac{F(w^{y_1}(t+\tau_{y_1}^{\varepsilon},x))
        - F(w^{y_2}(t+\tau_{y_2}^{\varepsilon},x))}
      {w^{y_1}(t+\tau_{y_1}^{\varepsilon},x)
        -w^{y_2}(t+\tau_{y_2}^{\varepsilon},x)},
      & \text{if }w^{y_1}(t+\tau_{y_1}^{\varepsilon},x)
      \neq w^{y_2}(t+\tau_{y_2}^{\varepsilon},x),
      \\ \xi(x), &\text{if } w^{y_1}(t+\tau_{y_1}^{\varepsilon},x)
      = w^{y_2}(t+\tau_{y_2}^{\varepsilon},x).
    \end{cases}
  \end{equation*}
  From the assumptions, it follows directly that $G$  is a bounded
  measurable function. Due to
  \cite[Theorem~7.4.1]{Freidlin1985functionalintegration}, there exists
  for $\PP$-a.e.~$\xi$ a unique classical solution to
  \eqref{eqn:cauchy_problem_difference_equation}. Moreover, since
  $w^{y_1}(0,x) = \bbone_{[y_1,\infty)}(x)$, it holds that
  \begin{equation}
    \label{eqn:Wtau}
    W(-\tau_{y_1}^{\varepsilon},x) =
    w^{y_1}(0,x)-w^{y_2}(\tau_{y_2}^{\varepsilon}-\tau_{y_1}^{\varepsilon},x)
    = \bbone_{[y_1,\infty)}(x)
    -w^{y_2}(\tau_{y_2}^{\varepsilon}-\tau_{y_1}^{\varepsilon},x).
  \end{equation}
  Together with the fact that $0<w^{y_i}(t,x)<1$ for $i = 1,2$ and for
  all $t>0$ and $x \in \R$ (cf.\ \eqref{eqn:sol01bd}),
  \eqref{eqn:Wtau} implies that $W(-\tau_{y_1}^{\varepsilon},x)<0$ if
  $x<y_1$ and $W(-\tau_{y_1}^{\varepsilon},x)>0$ if $x>y_1$. By
  Lemma~\ref{lem:sturmian_principle}, for all
  $t>-\tau_{y_1}^{\varepsilon}$, the sets $\{ x \in \R : W(t,x)>0\}$ and
  $\{ x \in \R : W(t,x)<0\}$ are intervals. But due to the continuity of
  $w^{y_1}$ and $w^{y_2}$, we also know that
  $W(0,0) = w^{y_1}(\tau_{y_1}^\varepsilon ,0)
  - w^{y_2}(\tau_{y_2}^\varepsilon ,0)= \varepsilon-\varepsilon = 0$.
  Therefore, the above reasoning supplies us with
  \begin{equation}
    \label{eqn:monotonicity_of_initial_condition}
    \begin{aligned}
      w^{y_1}(\tau_{y_1}^{\varepsilon},x)
      &\le w^{y_2}(\tau_{y_2}^{\varepsilon},x),
      \qquad &\text{ if $x<0$,}\\
      w^{y_1}(\tau_{y_1}^{\varepsilon},x)
      &\ge w^{y_2}(\tau_{y_2}^{\varepsilon},x),
      &\text{ if $x>0$.}
    \end{aligned}
  \end{equation}
  That is, the function $y \mapsto w^y(\tau_{y}^{\varepsilon},x)$ is
  non-decreasing if $x<0$ and non-increasing on $x>0$.  As a consequence,
  the limit
  $w^{\infty}_{\varepsilon}(0,x) := \lim_{y \to \infty} w^y(\tau_{y}^{\varepsilon},x)$
  exists pointwise, and thus locally uniformly, for all $x \in \R$,  and
  also implies $0\le w_{\varepsilon}^{\infty}(0,\cdot)\le 1$. As a
  consequence, the right-hand side of  \eqref{eqn:divergent_initial_value}
  converges locally uniformly for $t=0$. (This should be compared to
    \eqref{eqn:homogeneous_stretching} in the introduction, which
    describes the ``spatial stretching'' of re-centred solutions to the
    homogeneous F-KPP equation.)

  To prove that the local uniform convergence postulated in
  \eqref{eqn:divergent_initial_value}   holds true for $t>0$  also, one
  uses standard estimates on solutions of quasilinear parabolic equations
  (see, e.g., \cite{LSU68}, Chapter V). As a consequence of these
  estimates, the solutions $w^y(t,x)$ together with their derivatives are
  bounded locally uniformly in $(t,x)$, uniformly for all $y$
  sufficiently large. Hence the set $\{w^y:y\ge 0\}$ is pre-compact in
  $C^{1,2}_{\mathrm{loc}}(\mathbb R_+\times \mathbb R)$. It therefore contains
  converging sub-sequences, and every limit point of such a sub-sequence
  is a solution to \eqref{eqn:FKPP} with initial condition
  $w^\infty(0,\cdot)$. As the solution to \eqref{eqn:FKPP} with that
  given initial condition is unique, this implies that all subsequential
  limits must agree and thus \eqref{eqn:divergent_initial_value} holds for
  all $t>0$, as well as that $w^\infty$ solves \eqref{eqn:FKPP} for
  $t\ge 0$. We omit here the proof for $t<0$, as it will not be needed
  later on.
\end{proof}

The next corollary is a direct consequence of
Proposition~\ref{pro:divergent_initial_value}.  It formalises the idea
that when a solution to \eqref{eqn:FKPP} moves away from $0$, it
increases quickly to $1$. This is going to be relevant later on (cf.
  \eqref{eqn:tightintroca} in the introduction).

\begin{corollary}
  \label{cor:almost_surely_finite_time}
  For every  $\varepsilon \in (0,1/2)$ there exists a $\PP$-a.s.\ finite
  random variable $T = T(\xi)$ such that for all $y \in \R$ large
  enough, and any $t$ for which $w^y(t,0) = \varepsilon$, it holds that
  \begin{equation*}
      w^y(t+t',0)\ge 1-\varepsilon \quad \text{for all $t'\in [T,T+1]$.}
  \end{equation*}
\end{corollary}

\begin{proof}
  Let $y \in \R$ and $t\ge0$ be such that  $w^y(t,0) = \varepsilon$. By
  \eqref{eqn:temporal_quantile} and the finiteness of
  $ \tau_y^{\varepsilon}$ deduced below that display, there exists some
  $s_0 = s_0(y) \ge 0$ such that $t = \tau_y^{\varepsilon}+s_0$.

  Consider $w^{\infty}_{\varepsilon}$ from
  Proposition~\ref{pro:divergent_initial_value} and let
  \begin{equation*}
    s_1= \inf\{s>s_0 : w^{\infty}_{\varepsilon}(s',0) \ge 1-\varepsilon/2
      \text{ for all }s'>s \};
  \end{equation*}
  note that as $w_\varepsilon^{\infty}$ solves \eqref{eqn:FKPP}, it
  follows  by \cite[Theorem 7.6.1]{Freidlin1985functionalintegration}
  that for $\PP$-a.e.\ realisation of the environment,
  $\lim_{s \to \infty} w^{\infty}_{\varepsilon}(s,x) = 1$, and hence $s_1$
  is $\PP$-a.s.\ finite. Next, taking advantage of the fact that the
  convergence in Proposition~\ref{pro:divergent_initial_value} is locally
  uniform in $t$, due to the continuity of the functions involved and
  using the compactness of $[s_1, s_1+1]$, it holds for large enough
  $y\in \R$ that
  \begin{equation*}
    \sup_{s'\in [s_1,s_1+1]}|
    w^y(\tau_y^{\varepsilon}+s',0)-w^{\infty}_{\varepsilon}(s',0) |
    < \varepsilon/2.
  \end{equation*}
  Setting $T = s_1-s_0$, we thus obtain for all $y$ large enough and
  for all $t'\in [T,T+1]$ (with $s'=s_0+t'\in[s_1,s_1+1]$) that
  \begin{equation*}
    w^y(t+t',0) = w^y(\tau_y^{\varepsilon}+s',0)
    \ge w^{\infty}_{\varepsilon}(s',0)-\varepsilon/2 \ge  1-\varepsilon.
  \end{equation*}
  This completes the proof.
\end{proof}

This result concludes our analytic preparations on how the set of
zero-crossings of solutions to linear parabolic equations evolves, and of
how it can be applied to the difference of temporally shifted solutions
of \eqref{eqn:FKPP}.

\section{Tilting and exponential change of measure}
\label{sec:tilted_measure}

The next tool that we introduce is a change of measure for Brownian paths
in the Feynman-Kac representation, which makes certain large deviation
events  typical. These measures have been featured heavily in
\cite{CD2020,CDS2021,DS2022}  already, including in the proof of their
respective versions of Proposition~\ref{prop:perturbations}. In the
aforementioned articles, this change of measure has been employed so as to
make solutions to \eqref{eqn:PAM} amenable to the investigation by more
standard probabilistic tools. Here we go a step further and consider the
stochastic processes driving the tilted path measures. This, in turn, gives
us even more precise control on the tilted measures and allows for
comparisons with Brownian motion with constant drift, see Proposition
\ref{pro:process_driving_tilted_measure} below.

To define the tilted measures we set
\begin{equation}
  \label{eqn:definition_of_zeta}
  \zeta := \xi -\es.
\end{equation}
Due to the uniform boundedness \eqref{eqn:uniformly_bounded} it follows
that $\PP$-a.s.~for all $x \in \R$,
\begin{equation}
  \label{eqn:zetaboundedness}
  \zeta(x) \in [\ei-\es,0],
\end{equation}
and $\zeta$ is $\PP$-a.s.\ locally Hölder continuous with the same
exponent as $\xi$. Moreover, $\zeta$ also inherits the stationarity as
well as the mixing property from $\xi$.

For the Brownian motion $(X_t)_{t\ge0}$ under the measure $P_x$, as used
in the Feynman-Kac representations of Proposition~\ref{pro:Feynman-Kac},
we introduce first hitting times as
\begin{equation}
  \label{eqn:Hy}
  H_y := \inf \{ t \ge 0 : X_t = y\} \quad \text{for $y \in \R$.}
\end{equation}
Analogously to \cite{CD2020,CDS2021,DS2022}, we define for
$x\le y \in \mathbb R$ and $\eta < 0$ the \emph{tilted} path measures
characterised through events $A \in \sigma(X_{t\wedge H_y}, t \ge0)$ via
\begin{equation}
  \label{eqn:definition_of_tilted_measure}
  P_{x,y}^{\zeta,\eta}(A)
  := \frac{1}{Z_{x,y}^{\zeta,\eta}}
  E_x\Big[e^{\int_0^{H_y} (\zeta(X_s)+\eta ) \D s} ; A\Big],
\end{equation}
with normalising constant
\begin{equation}
  \label{eqn:tilted_measure_normalisation_constant}
  Z_{x,y}^{\zeta,\eta}
  := E_x\Big[e^{\int_0^{H_y} (\zeta(X_s)+\eta )\D s} \Big] \in (0,1].
\end{equation}
By the strong Markov property, it follows easily that the measures are
consistent in the sense that
$P_{x,y'}^{\zeta,\eta}(A) =P_{x,y}^{\zeta,\eta}(A)$ for $x\le y\le y'$
and $A \in \sigma(X_{t\wedge H_y},{t\ge0})$. Hence, for any $x\in \R$, we
can extend $P_{x,y}^{\zeta,\eta}$ to a
\begin{equation}
  \label{eqn:ext}
  \text{probability measure }P_{x}^{\zeta,\eta} \text{ on }\sigma(X_t,t\ge0)
\end{equation}
with the help of Kolmogorov's extension theorem. We write
$E_x^{\zeta,\eta}$ for the expectation with respect to the probability
measure $P_x^{\zeta,\eta}$.

Finally, as in \cite[(2.8)]{DS2022}, we introduce the annealed
logarithmic moment generating function
\begin{equation}
  \label{eqn:gf}
  L(\eta ) := \mathbb E[\ln Z_{0,1}^{\zeta ,\eta }],
\end{equation}
and for  denote by $\overline \eta (v)<0$ the unique solution of the
equation $L'(\overline \eta (v)) = \frac 1v$ for any $v > v_c$ (where
  $v_c$ is as in Assumption~\ref{ass:VEL}). Observe that $\overline \eta $
is well-defined as by \cite[Lemma~2.4]{DS2022},
\begin{equation}
  \label{eqn:lineeta}
  \parbox{12cm}{$\overline \eta (v)$ exists for every $v>v_c$;
    $v\mapsto \overline \eta (v)$ is a continuous decreasing function and
    $\lim_{v\to\infty }\overline \eta (v) = -\infty$.}
\end{equation}

The strong Markov property furthermore entails that, for a fixed
realisation $\zeta$ and any $\eta < 0$, the normalising constants
\eqref{eqn:tilted_measure_normalisation_constant} are
multiplicative in the sense that for any $x<y<z$ in $\R$,
\begin{equation}
  \label{eqn:multiplicativity}
  Z_{x,z}^{\zeta,\eta}  = Z_{x,y}^{\zeta,\eta} Z_{y,z}^{\zeta,\eta}.
\end{equation}
Defining, for some arbitrary but fixed $x_0\in \mathbb R$, the function
\begin{equation}
  \label{eqn:Zdef}
  Z^{\zeta ,\eta }(x) :=
  \begin{cases}
    ({Z_{x_0,x}^{\zeta ,\eta }})^{-1},\qquad&\text{if }x\ge x_0,
    \\ Z_{x,x_0}^{\zeta,\eta },&\text{if }x<x_0,
  \end{cases}
\end{equation}
the identity \eqref{eqn:multiplicativity} thus implies that for all $x<y$
we have
\begin{equation}
  \label{eqn:Zrel}
  Z^{\zeta,\eta }_{x,y} = \frac{Z^{\zeta ,\eta }(x)}{Z^{\zeta ,\eta }(y)}.
\end{equation}
The following lemma states some useful properties of the function
$Z^{\zeta, \eta }$.

\begin{lemma}
  \label{lem:Z}
  For every locally Hölder continuous function
  $\zeta :\mathbb R\to [-(\es-\ei),0]$ and $\eta < 0$, the function
  $Z^{\zeta ,\eta }$ is non-decreasing, strictly positive, twice
  continuously differentiable and satisfies
  \begin{equation}
    \label{eqn:ZODE}
    \frac 12 \Delta Z^{\zeta ,\eta }(x)
    + (\zeta(x) + \eta )Z^{\zeta ,\eta }(x) = 0, \qquad x\in \mathbb R.
  \end{equation}
  Furthermore,
  \begin{equation}
    \label{eqn:Zlnder}
    b^{\zeta ,\eta }(x) := \frac {\D}{\D x} \ln Z^{\zeta ,\eta }(x)
    \in \big[\underline{v}(\eta), \vbar(\eta)\big],
  \end{equation}
  where
  $\underline{v}(\eta) := \sqrt{2|\eta|}$ and
  $\vbar(\eta) := \sqrt{2(\es-\ei+|\eta|)}$.
\end{lemma}

\begin{proof}
  The monotonicity and the strict positivity of $Z^{\zeta ,\eta }$ follow
  directly from its definition~\eqref{eqn:Zdef}, using also
  \eqref{eqn:tilted_measure_normalisation_constant}.

  To show \eqref{eqn:ZODE},
  we observe that, for any interval $[x_1,x_2]$, the equation
  $\frac 12 \Delta u(x) + (\zeta (x)+\eta ) u(x) =0$, $x\in[x_1,x_2]$, with
  boundary conditions $u(x_i) = Z^{\zeta ,\eta }(x_i)$, $i=1,2$, has a
  unique classical solution (see, e.g., \cite[Corollary~6.9]{GT2001}).
  Denoting by $T$ the exit time of $X$ from $[x_1,x_2]$,
  this solution can be represented as (see \cite[Theorem~II(4.1), p.48]{Bass98})
  \begin{equation}
    u(x) = E_x\big[Z^{\zeta ,\eta }(X_T)
      e^{\int_0^T (\zeta (X_s)+\eta )\D s}\big].
  \end{equation}
  On the other hand, for $x\in[x_1,x_2]$, taking $y=x_2$ in
  \eqref{eqn:Zrel}, using
  \eqref{eqn:tilted_measure_normalisation_constant}, and the strong Markov
  property at time $T$,
  \begin{equation}
    \begin{split}
      Z^{\zeta ,\eta }(x)
      &=Z^{\zeta ,\eta }(y)Z^{\zeta ,\eta }_{x,y}
      \\&=Z^{\zeta ,\eta }(y)E_x\big[e^{\int_0^T (\zeta (X_s)+\eta )\D
          s}Z_{X_T,y}^{\zeta ,\eta }\big]
      \\&= E_x\big[Z^{\zeta ,\eta }(X_T)
        e^{\int_0^T (\zeta (X_s)+\eta )\D s}\big].
    \end{split}
  \end{equation}
  Therefore, $Z^{\zeta ,\eta }$ satisfies \eqref{eqn:ZODE} on $[x_1,x_2]$.
  Since the interval $[x_1,x_2]$ is arbitrary, \eqref{eqn:ZODE} holds for
  every $x\in \mathbb R$.

  To show \eqref{eqn:Zlnder}, note first that $b^{\zeta ,\eta }$
  is well-defined since $Z^{\zeta ,\eta }$ is strictly positive and
  differentiable, by \eqref{eqn:ZODE}. Therefore, with $y\ge x$, by
  \eqref{eqn:Zrel} and the strong Markov property again,
  \begin{align}
    \label{eqn:bx_derivative}
    \begin{split}
      b^{\zeta,\eta}(x)
      &=\frac{\D}{\D x}\ln Z^{\zeta,\eta}(x)
      =\frac{\D}{\D x}\ln Z_{x,y}^{\zeta,\eta}
      \\&= \lim\limits_{\varepsilon \to 0^+} \varepsilon^{-1}
      \Big(\ln E_{x}\Big[e^{\int_0^{H_y}(\zeta(X_s)+\eta)\D s}\Big]
        -\ln E_{x-\varepsilon}\Big[e^{\int_0^{H_y}(\zeta(X_s)+\eta)\D s} \Big]\Big)
      \\ &=  -\lim\limits_{\varepsilon \to 0^+} \varepsilon^{-1}
      \ln E_{x-\varepsilon}\Big[e^{\int_0^{H_x}(\zeta(X_s)+\eta)\D s} \Big].
    \end{split}
  \end{align}
  It is a known fact that for $\alpha>0$ and
  $z_1,z_2\in \R$, it holds that
  \begin{equation} \label{eqn:constPotGen}
  \ln E_{z_1}[e^{-\alpha H_{z_2}}] = -\sqrt{2\alpha}|z_1-z_2|
  \end{equation}
  (cf.~\cite[(2.0.1), p. 204]{Borodin1996HandbookOB}). In combination
  with the bounds \eqref{eqn:zetaboundedness}, the expectation on the
  right-hand side of \eqref{eqn:bx_derivative} thus satisfies
  \begin{equation}
  \label{eqn:ppp}
    \begin{split}
      -\varepsilon \sqrt{2|\eta|} &=
      \ln E_{x-\varepsilon}\big[e^{H_x\eta} \big]
      \ge \ln E_{x-\varepsilon}\Big[e^{\int_0^{H_x}(\zeta(X_s)+\eta)\D s} \Big]
      \\&\ge \ln E_{x-\varepsilon}\big[e^{H_x(\ei-\es+\eta)} \big]
      = -\varepsilon\sqrt{2(\es-\ei+|\eta |)},
    \end{split}
  \end{equation}
  which together with \eqref{eqn:bx_derivative} implies
  \eqref{eqn:Zlnder}.
\end{proof}

The function $b^{\zeta,\eta}(x)$ introduced in \eqref{eqn:Zlnder} is
useful in describing the law of $X$ under the tilted measure, as it
allows an interpretation of the tilted process as a Brownian motion with
an inhomogeneous drift, by constructing an appropriate SDE as follows.

\begin{proposition}
  \label{pro:process_driving_tilted_measure}
  Let $x_0 \in \R$, $\eta <0$ and let $\zeta: \R \to (-\infty,0]$ be a
  locally Hölder continuous function that is uniformly bounded from
  below. Furthermore, denote by $B$ a standard Brownian motion. Then the
  distribution of the solution to the SDE
  \begin{equation}
    \label{eqn:YSDE}
    \begin{split}
      \D X_t &= \D B_t + b^{\zeta,\eta}(X_t)\D t, \qquad
      t> 0,\\
      X_0 &= x_0,
    \end{split}
  \end{equation}
  agrees with $P_{x_0}^{\zeta,\eta}$.
\end{proposition}

\begin{proof}
  The proof is based on an exponential change of measure for diffusion
  processes. For the sake of simplicity, we write $b$ for $b^{\zeta,\eta }$
  and $Z$ for $Z^{\zeta ,\eta }$ whenever there is no risk of confusion. By
  \eqref{eqn:Zlnder} we obtain that
  \begin{equation}
    \label{eqn:bprime}
    b' = (\ln Z)'' = \Big(\frac {Z'}Z\Big)'= \frac{\Delta Z}{Z}-\Big(\frac
      {Z'}{Z}\Big)^2
    = -2(\zeta+\eta) -b^2.
  \end{equation}
  Therefore, the bounds \eqref{eqn:Zlnder} and
  \eqref{eqn:zetaboundedness} imply that $b$ is a bounded Lipschitz
  function and thus there is a strong solution to \eqref{eqn:YSDE}, whose
  distribution we denote by  $Q_{x_0}=Q_{x_0}^{\zeta ,\eta }$. Let
  further, as previously, $P_{x_0}$ be the distribution of Brownian
  motion started from $x_0$, and let $Q_{x_0}^{t}$ and $P_{x_0}^t$ be the
  restrictions of those distributions to the time interval $[0,t]$, $t>0$.
  As a consequence of the Cameron-Martin-Girsanov theorem (see, e.g.,
    \cite[Theorem V.27.1]{RogrsWilliams2000Markov} for a suitable
    formulation), it is well known that
  \begin{equation}
    \label{eqn:Mdef}
    \frac{\D Q_{x_0}^t}{\D P_{x_0}^t}=
     \exp\Big\{ \int_0^t b(X_s)\D X_s
      - \frac{1}{2}\int_0^t b^2(X_s)\D s\Big\}=:M_t,
  \end{equation}
  for a $P_{x_0}$-martingale $M$. (The fact that $M_t$ is a martingale
    follows, e.g.,~from \cite[Theorem~IV.37.8]{RogrsWilliams2000Markov},
    since $b$ is a bounded function.)

  With the aim of arriving at a comparison with
  \eqref{eqn:definition_of_tilted_measure}, we claim that
  \begin{equation}
    \label{eqn:MZrelation}
    M_t = \frac{Z(X_t)}{Z(X_0)}e^{\int_0^t(\zeta (X_s)+\eta )\D s}.
  \end{equation}
  To see this, note first
  that applying Itô's formula to $\ln Z(x) = \int_{x_0}^x b(t) \D t$ yields
  \begin{equation}
    \label{eqn:roughbound}
    \frac{Z(X_t)}{Z(X_0)}
    = \exp\Big\{\ln Z(X_t)- \ln Z(X_0)\Big\}
    = \exp\Big\{\int_0^t b(X_s)\D X_s + \frac{1}{2}\int_0^t b'(X_s)\D s\Big\}.
  \end{equation}
  Comparing this with \eqref{eqn:Mdef} shows that
  \begin{equation}
    M_t =\frac{Z(X_t)}{Z(X_0)} \exp\Big\{-\frac{1}{2}\int_0^t \big(b'(X_s)
      + b^2(X_s)\big)\D s \Big\},
  \end{equation}
  which together with \eqref{eqn:bprime} implies \eqref{eqn:MZrelation}.

  We can now complete the proof of the proposition. For the sake of
  clarity, we sometimes write expectations with respect to a probability
  measure $Q$ as $E^Q$. For $y\ge x_0$, let $Q_{x_0,y}$ be the measure
  $Q_{x_0}$ restricted to the $\sigma $-algebra
  $\mathcal H_y=\sigma (X_{s\wedge H_y}:s\ge 0)$. To show that
  $Q_{x_0} = P_{x_0}^{\zeta ,\eta }$, it is sufficient to show that
  $Q_{x_0,y} = P_{x_0,y}^{\zeta ,\eta }$ for all $y >x_0$ (see
    \eqref{eqn:definition_of_tilted_measure}). For this purpose, we
  observe that by Lemma~\ref{lem:Z}, $Z$ is a bounded function on
  $(-\infty,y]$ and thus the stopped martingale
  $M^{H_y}_t = M_{t\wedge H_y}$ is uniformly bounded from above.
  Therefore, by the optional stopping theorem, for any $A\in \mathcal H_y$,
  using \eqref{eqn:Mdef} for the second equality,
  \begin{equation*}
    \begin{split}
      Q_{x_0,y}(A)
      &= \lim_{t\to\infty} E^{Q_{x_0,y}}[\bbone_{A\cap \{H_y\le t\}}]
      = \lim_{t\to\infty} E^{P_{x_0}}[M_t\bbone_{A\cap \{H_y\le t\}}]
      \\&= \lim_{t\to\infty} E^{P_{x_0}}\big[E^{P_{x_0}}[M_t\bbone_{A\cap \{H_y\le
            t\}}\mid \mathcal H_y]\big]
      \\&= \lim_{t\to\infty} E^{P_{x_0}}\big[\bbone_{A\cap \{H_y\le
            t\}}E^{P_{x_0}}[M_t\mid \mathcal H_y]\big]
      \\&= \lim_{t\to\infty} E^{P_{x_0}}\big[\bbone_{A\cap \{H_y\le
            t\}}M_{H_y}\big] = E^{P_{x_0}}[M_{H_y} \bbone_A].
    \end{split}
  \end{equation*}
  By \eqref{eqn:MZrelation},
  $M_{H_y} = \frac{Z(y)}{Z(x_0)} e^{\int_0^{H_y}(\zeta (X_s)+\eta )\D s}$, and
  thus, also by \eqref{eqn:Zrel},
  \begin{equation*}
    Q_{x_0,y}(A)
    =  (Z_{x_0,y}^{\zeta ,\eta })^{-1}E_{x_0}[ e^{\int_0^{H_y}(\zeta (X_s)+\eta )
        \D s}\bbone_A] = P_{x_0,y}^{\zeta ,\eta }(A)
  \end{equation*}
  as required. This completes the proof.
\end{proof}

Proposition~\ref{pro:process_driving_tilted_measure} together with the
uniform bounds \eqref{eqn:Zlnder} on $b^{\zeta,\eta}$ allows for a
comparison between the tilted measures
\eqref{eqn:definition_of_tilted_measure} and Brownian motion with
constant drift. The next lemma makes this precise. For a given drift
$\alpha \in \R$, we write $P_x^{\alpha}$ for the law of Brownian motion
with constant drift $\alpha$ started at $x$ and $E_x^{\alpha}$ for the
corresponding expectation.

\begin{lemma}
  \label{lem:comparison_lemma}
  Let $\zeta: \R \to [-(\es-\ei),0]$ be locally Hölder continuous and let
  $\eta <0$. Then, for any starting point $x\in \R$ and any bounded
  non-decreasing function $g: \R \to \R$,
  \begin{equation*}
    E_{x}^{\underline{v}(\eta)}[g(X_t)]
    \le E_{x}^{\zeta,\eta}[g(X_t)] \le E_x^{\vbar(\eta)}[g(X_t)],
  \end{equation*}
  where $\underline{v}(\eta)$ and
  $\vbar(\eta)$ have been introduced in Lemma \ref{lem:Z}.
\end{lemma}

\begin{proof}
  By Proposition~\ref{pro:process_driving_tilted_measure}, the process
  $X_t$ driven by the tilted measure $P^{\zeta ,\eta }_{x_0}$ has
  generator
  $L^{\zeta,\eta} = \tfrac{1}{2} \Delta + b(x)\frac{\D}{\D x}$.
  Let further $L^{v} = \tfrac{1}{2} \Delta + v\frac{\D}{\D x}$ be the
  generator of the Brownian motion with drift $v$.
  Then, for any non-decreasing $g \in C^2_b(\R)$, it follows from
  \eqref{eqn:Zlnder} that
  \begin{equation*}
    L^{\underline{v}(\eta)}g \le L^{\zeta,\eta}g \le
    L^{\overline{v}(\eta)}g.
  \end{equation*}
  Since, by Kolmogorov's forward equation,
  $\frac {\D}{\D t} E_x^{\zeta ,\eta } [g(X_t)]
  = E_x^{\zeta ,\eta }[(L^{\zeta ,\eta }g)(X_t)]$
  and analogously for the measures
  $E_x^{\underline v}$ and $E_x^{\overline v}$, the statement of the
  lemma follows for any non-decreasing $g\in C_b^2(\mathbb R)$. The extension
  to arbitrary bounded non-decreasing functions $g$ follows by approximating $g$
  by a sequence of non-decreasing functions in $C^2_b(\R)$ and using the
  dominated convergence theorem.
\end{proof}

\section{Perturbations of the Feynman-Kac representation}
\label{sec:Perturbation_results_for_the_PAM}

In this section we provide a regularity type result for the Feynman-Kac
representation \eqref{eqn:Feynman-Kac_PAM} of solutions to the parabolic
Anderson model \eqref{eqn:PAM} with  initial conditions of Heaviside type.  A
variant of such results was developed in \cite{CDS2021,DS2022} (cf.
  Lemmas 3.11 and 3.13 from \cite{DS2022}, or Lemma 4.1 of
  \cite{CDS2021}) for the study of the fronts of \eqref{eqn:FKPP} and
\eqref{eqn:PAM}. In the current setting, the perturbation results will be
used together with the identity \eqref{eqn:FKPP_solution} and the Feynman-Kac
representation (Proposition~\ref{pro:Feynman-Kac}) in order to get bounds
on the solutions to~\eqref{eqn:FKPP} in the proof of the fundamental
Lemma~\ref{lem:contradiction} in
Section~\ref{sec:proofof_contradiction_lemma}.

To avoid the dependence of various constants appearing in these
perturbation results on the speed, we assume for the rest of the article
that the speeds we allow are contained in some arbitrary but fixed
compact interval $V \subset (v_c,\infty)$ which has $v_0$ in its interior
(recall that we require \eqref{eqn:VEL} to hold). As we can otherwise
choose $V$ arbitrarily large, this does not impose any further
restrictions for what follows.
With $\overline \eta (v)<0$ as in \eqref{eqn:lineeta}, we further fix a compact interval
$\triangle\subset(-\infty,0)$ such that
\begin{equation}
  \label{eqn:baretaub}
  \text{the set }\{\overline \eta (v):v\in V\}
  \text{ is contained in the interior of }\triangle.
\end{equation}
From now on, all newly introduced constants may depend implicitly on
the compact sets $V$ and $\triangle$ and on the bounds on the
environment $\es$ and $\ei$. Their dependence on other parameters is made
explicit when the constants are introduced. In order to facilitate
reading, for quantities that depend on the realisation of the environment
$\xi $ we use letters $\mathcal T$, $\mathcal N$, etc.

\begin{proposition}
  \phantomsection
  \label{prop:perturbations}
  \begin{enumerate}
    \item \label{timeperturbation}
    For every $\delta>0$ and $A>0$, there exist a constant
    $\Cl{c:timePert}= \Cr{c:timePert}(A,\delta ) \in (1,\infty)$ and a
    $\PP$-a.s.~finite random variable
    $\mathcal{T}_1= \mathcal T_1(A, \delta )$ such that for all
    $t\ge \mathcal{T}_1$, uniformly in $0\le h \le t^{1-\delta}$ and
    $x,y\in [-At,At]$ with $x<y$, $\frac{y-x}t \in V$ and
    $\frac{y-x}{t+h}\in V$,
    \begin{align*}
      E_{x}&\Big[e^{\int_0^{t + h}\xi(X_s)\D s} ;
        X_{t+h} \ge y \Big]
      \le \Cr{c:timePert}e^{\Cr{c:timePert} h}
      E_{x}\Big[e^{\int_0^{t}\xi(X_s)\D s} ; X_t \ge y \Big].
    \end{align*}

    \item \label{spaceperturbation}
    Let $\delta: (0,\infty) \to (0,\infty)$ be a function tending to $0$
    as $t  \to \infty$, and let $A>0$. Then there exists a constant
    $\Cl{c:spacePert} = \Cr{c:spacePert}(A,\delta ) \in (1,\infty)$ and a
    $\PP$-a.s.~finite random variable $\mathcal{T}_2$ such that for all
    $t \ge \mathcal{T}_2$, uniformly in $0\le h \le t\delta(t)$ and
    $x,y\in [-At,At]$ with $x<y$, $\frac{y-x}{t} \in V$ and
    $\frac{y+h-x}{t}\in V$,
    \begin{align*}
      E_{x}&\Big[e^{\int_0^{t}\xi(X_s)\D s} ;
        X_t \ge y+h \Big]\le
      \Cr{c:spacePert}e^{- h/\Cr{c:spacePert}}
      E_{x}\Big[ e^{\int_0^{t}\xi(X_s)\D s} ; X_t \ge y \Big].
    \end{align*}
  \end{enumerate}
\end{proposition}

The proof of this proposition involves comparing the Feynman-Kac
representation \eqref{eqn:Feynman-Kac_PAM} to functionals with respect to
the family of tilted probability measures that were presented in
Section~\ref{sec:tilted_measure}. It is a rather straightforward, but
lengthy adaptation of parts of the proofs of Lemma~3.11(b) of
\cite{DS2022} (which corresponds to part \ref{timeperturbation} of
  Proposition~\ref{prop:perturbations}) and Lemma~4.1(b) of
\cite{CDS2021} (corresponding to part \ref{spaceperturbation} of
  Proposition~\ref{prop:perturbations}). There are two key differences in
the statements of our Proposition~\ref{prop:perturbations} when compared
to the two respective statements in \cite{CDS2021,DS2022}, that need to
be addressed:

\begin{enumerate}[label=(\Alph*)]
  \item
  \label{dif:A}
  Proposition~\ref{prop:perturbations} requires that its estimates hold
  uniformly over the ``starting point'' $x$ and the ``target point'' $y$
  in an interval growing linearly with time $t$. In the original
  statements, the starting point satisfies $x=vt$ and the target point
  essentially always corresponds to the origin.

  \item
  \label{dif:B}
  Proposition~\ref{prop:perturbations}\ref{spaceperturbation} involves a
  perturbation by the end point (that is, $y$ changes to $y+h$), while
  the starting point is perturbed in the original statement.
\end{enumerate}

In addition, \cite{CDS2021,DS2022} always consider travelling waves
going ``from left to right'', while for our purposes it is more suitable to
work with waves going ``from right to left''. This difference is easily dealt
with by mirroring the environment and we do not discuss it further.

Proving Proposition~\ref{prop:perturbations} thus requires checking that
these two differences can be dealt with by the original arguments. In the
following we provide the arguments of the proof, and furthermore describe
key locations where the arguments of \cite{CDS2021,DS2022} have to be
adapted. For auxiliary results from the original sources
\cite{CDS2021,DS2022} that can be employed without any further
adaptation, we will just refer to the respective references.

\smallskip

We start by introducing certain logarithmic moment generating functions
that are featured heavily in the proof. For $\eta <0$ and $y>x$, let
\begin{align}
  \label{eqn:logMomGen}
  \begin{split}
    \overline L_{x,y}^{\zeta}(\eta)
    &:=\frac{1}{y-x}\ln E_{x}\Big[ \exp\Big\{ \int_0^{H_{y}}\big( \zeta(X_s)
          +\eta \big)\D s \Big\} \Big]
    = \frac{1}{y-x}\ln (Z_{x,y}^{\zeta ,\eta }) ,
    \\L(\eta)&:=\E\big[ \overline L_{0,1}^{\zeta}(\eta) \big].
  \end{split}
\end{align}
(Note that the function $L(\eta )$  coincides with that in
\eqref{eqn:gf}.)
In order to control these functions, a simple but recurrently used
generalisation of \cite[Lemma~A.1]{DS2022} is given by the following
statement.

\begin{claim}
  \phantomsection
  \label{cl:A1}
  \begin{enumerate}
    \item For every $x<y$, the functions $L$ and
    $\overline{L}_{x,y}^\zeta$ are infinitely differentiable on
    $(-\infty,0)$ with
    \begin{align}
      \label{eqn:derRep}
      \big( \overline{L}_{x,y}^\zeta \big)'(\eta)
      &= \frac{1}{y-x}E_{x}^{\zeta,\eta}\big[ H_y \big],
      \\ \label{eqn:derDerRep}
      (\overline L_{x,y}^{\zeta})''(\eta)&=
      \frac{1}{y-x}\var_{P^{\zeta, \eta }_x}(H_y)
    \end{align}
    (where $\var_{P^{\zeta, \eta }_x}$ denotes  the variance with respect
      to $P_x^{\zeta ,\eta }$ defined in \eqref{eqn:ext}).

    \item
    There exists a constant $\Cl{constDerLMG}$ such that,
    $\mathbb P$-a.s.,
    \begin{equation}
      -\Cr{constDerLMG}
      \leq \inf_{\eta\in\triangle,|x-y|\geq 1}
      \big\{ \overline{L}_{x,y}^\zeta(\eta),L(\eta)  \big\}
      \leq \sup_{\eta\in\triangle,|x-y|\geq 1}
      \big\{ \overline{L}_{x,y}^\zeta(\eta),L(\eta) \big\}
      \leq -\Cr{constDerLMG}^{-1},
    \end{equation}
    and, for $n\in \{1,2\}$, the $n$-th derivatives satisfy
    \begin{align}
      \label{eqn:LderBd}
      \begin{split}
        \Cr{constDerLMG}^{-1}
        &\leq \inf_{\eta\in\triangle,|x-y|\geq 1}
        \big\{
          (\overline{L}_{x,y}^\zeta)^{(n)}(\eta),L^{(n)}(\eta) \big\}
        \\&\leq \sup_{\eta\in\triangle,|x-y|\geq 1}
        \big\{
          (\overline{L}_{x,y}^\zeta)^{(n)}(\eta),L^{(n)}(\eta) \big\}
        \leq \Cr{constDerLMG}.
      \end{split}
    \end{align}
  \end{enumerate}
\end{claim}

Since upon replacing the index $x$ by indices $x,y$, the proof of this
claim follows that of \cite[Lemma A.1(b)]{DS2022} verbatim, we omit it
here. (In particular, it only uses the upper and lower bounds of the
  potential $\zeta$, and otherwise is independent of it.)

To prove Proposition~\ref{prop:perturbations}, we are now interested in
the (random) tilting parameter $\eta_{x,y}^\zeta(v)$ for which the mean
speed of a particle on its way from $x$ to $y$ under the tilted measure
$P_x^{\zeta ,\eta }$ is precisely $v$, that is
\begin{equation}
  \label{eqn:etazeta}
  E_x^{\zeta ,\eta_{x,y}^\zeta (v)}[H_y] = \frac {y-x}v,\qquad v>0, x<y.
\end{equation}
(If no such parameter exists, we set $\eta_{x,y}^\zeta (v) =0$.) The next
lemma shows that, for $y-x$ large, a unique $\eta_{x,y}^\zeta(v)$
satisfying \eqref{eqn:etazeta} exists with high probability and that it
is close to $\overline \eta(v) $ (recall \eqref{eqn:lineeta}). It is an
extension of Lemma~2.5 of \cite{DS2022} and it is the first step on the
way to dealing with the difference \ref{dif:A} in the above list.

\begin{lemma}
  \phantomsection
  \label{lem:etazeta}
  \begin{enumerate}
    \item \label{item:etazetaA} For every $A>1$ there exists a
    $\PP$-a.s.~finite random variable $\mathcal N = \mathcal N(A)$ such that
    for all $v\in V$ and $x<y\in \mathbb R$ such that $y-x\ge \mathcal N$
    and $|x|,|y|\le A(y-x)$, the solution $\eta_{x,y}^{\zeta }(v)$ to
    \eqref{eqn:etazeta} exists and satisfies
    $\eta_{x,y}^{\zeta }(v)\in \triangle$.

    \item \label{item:etazetaB}
    For each $q\in \mathbb N$ there exists a constant
    $\Cl{c:etaConc}=\Cr{c:etaConc}(q)\in (0,\infty)$
    such
    that for all $n\ge C_4$
    \begin{equation} \label{eqn:compBC}
      \PP\bigg(\sup_{v\in  V} \sup_{x\in [-n,-n+1]} \sup_{y\in [0,1]}
        |\eta_{x,y}^\zeta (v) - \overline \eta (v)|
        \ge \Cr{c:etaConc}\sqrt{\frac{\ln n}{n}}\bigg) \le \Cr{c:etaConc} n^{-q}.
    \end{equation}
  \end{enumerate}
\end{lemma}

\begin{proof}[Proof (outline)]
  The proof follows the strategy of \cite[Lemma 2.5]{DS2022}: Item
  \ref{item:etazetaA} follows directly from \ref{item:etazetaB}, using
  the Borel-Cantelli lemma and \eqref{eqn:baretaub}, with the help of an
  additional union bound to take care of the uniformity in $y$.

  For the proof of \ref{item:etazetaB}, we remark that in
  \cite[Lemma~2.5(b)]{DS2022} it was shown that for all $n \ge C$
  \begin{equation} \label{eqn:simpBC}
    \PP\bigg(\sup_{v\in V} \sup_{x\in [-n,-n+1]}
      |\eta_{x,0}^\zeta (v) - \overline \eta (v)|
      \ge C\sqrt{\frac{\ln n}{n}}\bigg) \le Cn^{-q}.
  \end{equation}
  Hence, in order to derive \eqref{eqn:compBC}, compared to
  \eqref{eqn:simpBC}, we now have to include the additional supremum over
  $y\in[0,1]$. We first observe that replacing $n$ by $n+1$ in
  \eqref{eqn:simpBC} and using the shift invariance of $\zeta $ implies
  that \eqref{eqn:simpBC} also holds when
  $\eta_{x,0}^\zeta (v)$ is replaced by $\eta_{x,1}^\zeta (v)$. By
  \eqref{eqn:derRep} and \eqref{eqn:etazeta},
  $\eta_{x,y}^\zeta (v)$ solves
  \begin{equation}
    \label{eqn:etainbarL}
    \overline L_{x,y}'(\eta_{x,y}^\zeta (v)) = v^{-1}.
  \end{equation}
  From \eqref{eqn:derRep} it follows that
  $y\mapsto (y-x) (\overline L_{x,y}^\zeta )'(\eta )$ is increasing.
  Hence, for $y\in [0,1]$,
  \begin{equation}
    \label{eqn:monoarg}
    (0-x) \overline L_{x,0}'(\eta_{x,y}^\zeta (v))
    \le (y-x)
    \overline L_{x,y}'(\eta_{x,y}^\zeta (v))
    \le
    (1-x)\overline L_{x,1}'(\eta_{x,y}^\zeta (v)).
  \end{equation}
  The first inequality and \eqref{eqn:etainbarL} then imply
  \begin{equation*}
    \overline L_{x,0}'(\eta_{x,y}^\zeta (v))
    \le \frac{y-x}{v(0-x)} \le
    \frac 1 v \Big(1+\frac {1}{|x|}\Big).
  \end{equation*}
  By \eqref{eqn:LderBd}, $\overline L_{x,0}'$ is increasing and, on
  $\triangle$, its derivative is bounded from below by
  $\Cr{constDerLMG}^{-1}$. Hence, on the event that
  $\eta_{x,0}^{\zeta }(v)$ is close to $\overline \eta (v)$ (which is
    typical for $n$ large, by \eqref{eqn:simpBC}), this implies
  \begin{equation*}
    \eta_{x,y}^\zeta (v) \le \eta_{x,0}^\zeta(v)  + \frac
    {\Cr{constDerLMG}}{v|x|}.
  \end{equation*}
  Similarly, the second inequality in \eqref{eqn:monoarg} implies
  \begin{equation*}
    \eta_{x,y}^\zeta (v) \ge \eta_{x,1}^\zeta(v)  - \frac
    {\Cr{constDerLMG}}{v|x|}.
  \end{equation*}
  Combining these two inequalities with \eqref{eqn:simpBC} (and its
    version for $\eta_{x,1}^\zeta (v)$, as observed above) then yields
  \eqref{eqn:compBC} with $\Cr{c:etaConc} = 2C$ for $n$ large, and by
  adjusting $\Cr{c:etaConc}$ for all $n \in \mathbb N$.
\end{proof}

We next adapt Lemma~2.7~of \cite{DS2022} which provides a
spatial perturbation result for the tilting parameters
$\eta_{x,y}^\zeta(v)$ introduced in \eqref{eqn:etazeta}.

\begin{lemma}
  \label{lem:etapert}
  For $A>1$ let $\mathcal N =\mathcal N(A)$ be as in
  Lemma~\ref{lem:etazeta}.
  There exists a constant
  $\Cl{c:diffEmpLogMomGenFct}$
  such that for all $x,y\in \mathbb R$
  with $y-x\ge \mathcal N$ and $|x|,|y|\le A(y-x)$, $v\in V$, and
  $h\in [1,y-x]$, we have
  \begin{equation}
    \big | \eta_{x,y}^{\zeta }(v)- \eta_{x,y+h}^\zeta (v)\big|
    \le  \frac {\Cr{c:diffEmpLogMomGenFct}h}{y-x}.
  \end{equation}
\end{lemma}

\begin{proof}[Proof (outline)]
  The strategy for proving Lemma \ref{lem:etapert} is as follows: By
  Lemma~\ref{lem:etazeta},
  $\eta_{x,y}^\zeta(v), \eta_{x,y+h}^\zeta (v)\in\triangle$ for all
  $x,y,v,h$ as in the assumptions.  In particular, this means that
  $\eta_{x,y}^\zeta(v)$ and $ \eta_{x,y+h}^\zeta (v)$ are implicitly
  defined via \eqref{eqn:etazeta} or, equivalently,
  \eqref{eqn:etainbarL}. Therefore, the proof then proceeds by showing
  that there exists $\Cl{c:pert}$ such that
  \begin{equation}
    \label{eqn:eta_nPert}
    \sup_{\eta\in\triangle} \big| \big( \overline{L}_{x,y+h}^\zeta  \big)'(\eta)
    - \big( \overline{L}_{x,y}^\zeta  \big)'(\eta) \big|
    \leq \Cr{c:pert}\frac{h}{y-x},
  \end{equation}
  which will finish the proof, since again one can use the regularity of
  the functions $\overline{L}_{x,y}^\zeta$ and their derivatives to
  deduce the closeness of their arguments from their function values
  being close.

  It remains to establish \eqref{eqn:eta_nPert}. For this purpose note
  that using \eqref{eqn:derRep} and the strong Markov property at time
  $H_y$, one rewrites
  \begin{align*}
     \big( \overline{L}_{x,y+h}^\zeta  \big)'(\eta)
    - \big( \overline{L}_{x,y}^\zeta  \big)'(\eta)
    &=-\frac{h}{y-x+h}\big( \overline{L}_{x,y}^\zeta  \big)'(\eta)
    + \frac{h}{y-x+h} \big(\overline L_{y,y+h}^\zeta\big)'(\eta ).
  \end{align*}
  Claim \eqref{eqn:eta_nPert} then readily follows using the  bounds for
  $(\overline{L}_{x,y}^\zeta  )'$ from \eqref{eqn:LderBd}.
\end{proof}

Next, we introduce two auxiliary processes which we later relate to the
expressions appearing in Proposition~\ref{prop:perturbations}. We
consider, for $x <  y\in \mathbb R$ and $v>0$, the quantities
\begin{equation}
  \begin{split}
    Y^\approx_v(x,y) &:= E_x \Big[ e^{\int_0^{H_y}\zeta(X_s) \D s};
      H_y\in \Big[ \frac{y-x}{v} - K ,\frac{y-x}v\Big]\Big],
    \\Y^>_v(x,y) &:= E_x \Big[ e^{\int_0^{H_y}\zeta(X_s) \D s};
      H_y<\frac{y-x}{v} - K\Big],
  \end{split}
\end{equation}
where $K>0$ is a large constant that will be fixed later on in
\eqref{eqn:K} below. For this suitable choice of $K$ the quantities
$Y^{\approx}_v(x,y)$ and $Y^<_{v}(x,y)$ are comparable uniformly in the
admissible choices of $x$ and $y$. This result is a partial extension of
\cite[Proposition 3.5]{DS2022}.

\begin{lemma}
  \label{lem:Ycomparability}
  For $A>1$, let $\mathcal N =\mathcal N(A)$ be as in
  Lemma~\ref{lem:etazeta}. Then there exists a constant
  $\Cl{c:approxSmall} $
  such that for all $v\in V$ and all $x<y\in \mathbb R$ such that
  $y-x\ge \mathcal N$ as well as $|x|,|y|\le A(y-x)$, we have
  \begin{equation}
    \frac {Y^{\approx}_v(x,y)}{Y^{<}_v(x,y)}
    \in [\Cr{c:approxSmall}^{-1},\Cr{c:approxSmall}].
  \end{equation}
\end{lemma}

\begin{proof}
  The proof of this lemma contains a computation that is essential for a
  step in the proof of Lemma~\ref{lem:contradiction}, and is also
  featured in Section~\ref{sec:proofof_contradiction_lemma} below. We
  assume that $x,y$ satisfy the assumptions of the lemma, and, in order
  to keep the notation simple, we in addition assume that
  $x,y\in \mathbb Z$ (we refer to ~\cite[Section~1.9]{DS2022} for the
    technical issue and notation of how to deal with the case of
    non-integer $x$ and $y$). We write $\eta^*: = \eta_{x,y}^{\zeta }(v)$
  for conciseness of notation in the following and define
  \begin{equation}
    \sigma^* = \sigma_{x,y}^\zeta (v) := |\eta^* |
    \sqrt {\var_{P^{\zeta, \eta^* }_x}(H_y)}.
  \end{equation}
  We observe that $\mathbb P$-a.s., uniformly in $v\in V$,
  \begin{equation}
    \label{eqn:sigmaas}
    \Cl{c:varconst}\sqrt{y-x}
    \le \sigma_{x,y}^{\zeta  }(v)  \le \Cr{c:varconst}\sqrt {y-x}.
  \end{equation}
  Indeed, this follows directly from \eqref{eqn:derDerRep} and
  \eqref{eqn:LderBd}, with $\Cr{c:varconst}$ depending only on
  $\Cr{constDerLMG}$ and $\triangle$.

  Let further $\tau_z = H_{z}-H_{z-1}$, $z\in [x+1,y]\cap \mathbb Z$, and
  let $\widehat {\tau}_z := \tau_z-  E_x^{\zeta ,\eta^* }[\tau_z]$. Then,
  by the definition of $\eta^*$, for $x,y$ satisfying the assumptions of
  Lemma \ref{lem:etazeta}, we have
  $E_x^{\zeta ,\eta^* }[H_y]=\frac{y-x}{v}$. With this notation, a
  straightforward computation yields
  \begin{equation}
    \begin{split}
      Y_{v}^\approx(x,y)&=
      E_x\Big[ e^{ \int_{0}^{H_{y}}(\zeta(X_s)+\eta^*)\D s }
        \, e^{ -\eta^*\sum_{z=x+1}^{y}\widehat{\tau}_{z} };
        \, \sum_{i=x+1}^{y}\widehat{\tau}_{z}\in[ -K,0 ] \Big]
      e^{ -(y-x)\eta^*/v}
      \\ &=
      E_x^{\zeta,\eta^*}\bigg[
        e^{ -\sigma^* \frac{\eta^*}{\sigma^*}\sum_{z=x+1}^y\widehat{\tau}_{z} };
        \,  \frac{\eta^*}{\sigma^*}\sum_{z=x+1}^y\widehat{\tau}_{z}\in
        \Big[ 0,-\frac{K\eta^*}{\sigma^*} \Big] \bigg]
      e^{ -(y-x)( \frac{\eta^*}{v} - \overline{L}_{x,y}^\zeta(\eta^*)) }.
    \end{split}
  \end{equation}
  Defining $\mu_{x,y}^{\zeta ,\eta^* }$ to be the distribution of
  $\frac{\eta^*}{\sigma^*}\sum_{z=x+1}^y \widehat{\tau}_z$ under
  $P_x^{\zeta ,\eta^* }$, this implies
  \begin{equation}
    \label{eqn:Yap}
    Y_v^\approx(x,y)
    =e^{-(y-x)( \frac{\eta^*}{v} - \overline{L}_{x,y}^\zeta(\eta^*) )}
    \int_0^{\frac{-K\eta^*}{\sigma^*}}
    e^{-\sigma^* u} \mu_{x,y}^{\zeta,\eta^* }(\D u).
  \end{equation}
  A completely analogous computation then shows that
  \begin{equation}
    \label{eqn:Yle}
    Y_{v}^<(x,y)
    =e^{-(y-x)( \frac{\eta^*}{v} - \overline{L}_{x,y}^\zeta(\eta^*) )}
    \int_{\frac{-K\eta^*}{\sigma^*}}^{\infty}
    e^{-\sigma^* u} \mu_{x,y}^{\zeta,\eta^* }(\D u).
  \end{equation}

  The upshot of these computations is that under $P_{x}^{\zeta ,\eta^* }$,
  the random variables $\widehat {\tau}_z$, $z=x+1,\dots,y$, are centred,
  independent, have uniform exponential moments, and
  $\mu_{x,y}^{\zeta ,\eta^* }$ has unit variance.
  As a consequence,
  we can  uniformly approximate $\mu_{x,y}^{\zeta ,\eta^* }$ by the
  standard Gaussian measure $\Phi$, and the integrals appearing on the
  right-hand side of \eqref{eqn:Yap} and \eqref{eqn:Yle} are $\PP$-a.s.~both
  of order $(y-x)^{-1/2}$, uniformly in the $\zeta $  and $v\in V$ under
  consideration, and  for all $x,y$ satisfying the assumptions of
  Lemma~\ref{lem:etazeta}. More precisely, the conditions of the error
  estimate \cite[Theorem~13.3]{Bhattacharya1986nonrmal} for normal
  approximations are satisfied and we can hence
  apply~\cite[(13.43)]{Bhattacharya1986nonrmal} to deduce that
  \begin{equation}
    \label{eqn:normalapp}
    \sup_{\mathcal{C}}\big|
    \mu_{x,y}^{\zeta,\eta^*}(\mathcal C) - \Phi( \mathcal C)\big|
    \le \Cl{c:normErr} (y-x)^{-1/2},
  \end{equation}
  where the supremum is over all intervals in $\R$ and the constant
  $\Cr{c:normErr}>0$ only depends on the uniform bound of the exponential
  moments of the $\widehat \tau_z$'s. Now, due to \eqref{eqn:sigmaas}, we
  can choose $K>0$ large enough so that for some constants
  $\Cr{c:normErr}<\Cl{c:LBnormErr}= \Cr{c:LBnormErr}(K)
  <\Cl{c:UBnormErr} = \Cr{c:UBnormErr}(K)$,
  for all $x,y$ with $y-x\ge \mathcal N$ and $v\in V$ we have
 \begin{equation}
    \label{eqn:K}
    \Cr{c:LBnormErr} (y-x)^{-1/2} \le
    \Phi\big( [0,-K\eta^*/\sigma] \big) \le \Cr{c:UBnormErr} (y-x)^{-1/2},
  \end{equation}
  and thus infer
  \begin{equation*}
    (\Cr{c:LBnormErr}-\Cr{c:normErr})(y-x)^{-1/2}
    \le \mu_{x,y}^{\zeta, \eta^*}([0,-K\eta^*/\sigma])
    \le (\Cr{c:UBnormErr}+\Cr{c:normErr}) (y-x)^{-1/2}.
  \end{equation*}
  Plugging these estimates into \eqref{eqn:Yap} and \eqref{eqn:Yle}, we
  can then show that the integrals in these displays are both of the same
  order $(y-x)^{-1/2}$, following exactly the same steps as in
  \cite[Lemma~3.6]{DS2022}.
\end{proof}

Lemma~\ref{lem:Ycomparability} has an important corollary allowing to
approximate the Feynman-Kac formula for the PAM
(cf.~\eqref{eqn:Feynman-Kac_PAM}) by expressions involving
$Y^\approx_{v}(x,y)$. This approximation will also be used  in
Section~\ref{sec:proofof_contradiction_lemma} below. Its proof is
inspired by that of Lemma 3.7 in \cite{DS2022}.

\begin{lemma}
  \label{lem:YaproxFK}
  For each $A >1$, with $\mathcal N = \mathcal N(A)$ as in
  Lemma~\ref{lem:etazeta}, there exists a constant
  $\Cl{c:approxCst} = \Cr{c:approxCst}(K)\in (1,\infty)$
  such that for all $t\in \mathbb (0,\infty)$ and all $x<y\in \mathbb R$
  such that $y-x\ge \mathcal N$, $|x|,|y|\le A(y-x)$ and
  $\frac{y-x}{t}\in V$,
  \begin{equation} \label{eqn:approxEst}
    \Cr{c:approxCst}^{-1} Y^\approx_{v}(x,y)
    \le E_x\big[e^{\int_0^t \zeta (X_s)\D s}; X_t
      \ge y\big]
    \le \Cr{c:approxCst} Y^\approx_{v}(x,y).
  \end{equation}
\end{lemma}

\begin{proof}
  Since for $X_t$ starting in $x$ we have
  $\{X_{t}\geq y\}\subset\{ H_y\leq t \}$ and $\zeta\leq0$, we get
  $E_{x}\big[ e^{\int_0^{t}\zeta(X_s)\D s}; X_{t}\leq y \big]
  \leq (1+\Cr{c:approxSmall})Y_v^\approx(x,y)$
  by Lemma~\ref{lem:Ycomparability} and thus the last inequality in
  \eqref{eqn:approxEst} is obtained.

  To show the first inequality, we define for $\delta > 0$ and
  $y\in \mathbb R$ the random functions
  \begin{equation*}
    p_y(s):=E_y\big[ e^{\int_0^s\zeta(X_r)\D r};X_s\in[y,y+\delta ] \big]
  \end{equation*}
  which, almost surely, for all $s\in[0,K]$, are bounded from below by a deterministic constant
  $\Cl[small]{c:LBh} = \Cr{c:LBh}(K,\delta)\in (0,1)$ .
  Using the strong Markov property at $H_y$, we finally get
  \begin{align*}
    Y_v^\approx(x,y)
    &= E_{x}\left[ e^{\int_0^{H_y}\zeta(X_r) \D r};
      H_y\in\left[ t-K,t \right] \right] \\
    &\leq \Cr{c:LBh}^{-1}
    E_x\left[ e^{\int_0^{H_y}\zeta(X_r)  \D r}
      p_y(t-H_y); H_y\in\left[ t-K,t \right] \right] \\
    &\leq  \Cr{c:LBh}^{-1} E_{x}\left[ e^{\int_0^{H_y}\zeta(X_r)
        \D r}p_y(t-H_y)\right]\\
    &=\Cr{c:LBh}^{-1}
    E_{x}\left[ e^{\int_0^{t}\zeta(X_r)\D r};X_{t}\in[y,y+\delta ] \right].
  \end{align*}
  and the claim follows by choosing
  $\Cr{c:approxCst}:=\Cr{c:LBh}^{-1}\vee (1+\Cr{c:approxSmall})$.
\end{proof}

With this, we have made all the necessary extensions to the results in
\cite{CDS2021,DS2022} which are needed in order to accommodate for the
differences outlined in \ref{dif:A} and \ref{dif:B} and are equipped to
show Proposition~\ref{prop:perturbations}.

\begin{proof}[Proof of Proposition~\ref{prop:perturbations}]
  We denote $v:=(y-x)/t$, $v':=(y-x)/(t+h)$ and observe that by
  Lemma~\ref{lem:YaproxFK}, for $x$, $y$, $t$ and $h$ as in the
  statement, and by choosing $\mathcal T_1$ sufficiently large so that
  $y-x\ge \mathcal N$,
  \begin{equation}
    \label{eqn:AA0}
    \frac{E_{x}\Big[e^{\int_0^{t + h}\xi(X_s)\D s} ;
        X_{t+h} \ge y \Big]}
    { E_{x}\Big[e^{\int_0^{t}\xi(X_s)\D s} ;
        X_t \ge y \Big]} \le \Cr{c:approxCst}^2
    \frac{Y^{\approx}_{v'}(x,y)}{Y^{\approx}_{v}(x,y)}.
  \end{equation}
  The fraction on the right-hand side can be rewritten with the help of
  \eqref{eqn:Yap}. Using also the fact that the integral appearing in
  \eqref{eqn:Yap} is of order $(y-x)^{-1/2}$ uniformly in $v\in V$ and
  $y-x\ge \mathcal N$, as explained at the end of
  the proof of Lemma~\ref{lem:Ycomparability}, we obtain
  \begin{equation}
    \label{eqn:AA1}
    \frac{Y^{\approx}_{v'}(x,y)}{Y^{\approx}_{v}(x,y)}
    \le
    \Cl{c:apprRat}  \frac{\exp\big\{-(y-x)\big( \frac{\eta_{x,y}^\zeta(v') }{v'}
          - \overline{L}_{x,y}^\zeta(\eta_{x,y}^\zeta(v'))
          \big)\big\}}
    {\exp\big\{-(y-x)\big( \frac{\eta_{x,y}^\zeta(v) }{v}
          - \overline{L}_{x,y}^\zeta(\eta_{x,y}^\zeta(v)) \big)\big\}},
  \end{equation}
  for some finite constant $\Cr{c:apprRat} = \Cr{c:apprRat}(K)$.
  Denoting for any $\eta<0$
  \begin{equation}
    \label{eqn:Sdef}
    S_{x,y}^{\zeta ,v}(\eta ) := (y-x)\Big( \frac{\eta }{v}
      - \overline{L}_{x,y}^\zeta(\eta) \Big),
  \end{equation}
  the logarithm of the fraction on the right-hand side of \eqref{eqn:AA1}
  can be written as
  \begin{equation}
    \label{eqn:AA2}
    \big(S_{x,y}^{\zeta,v}(\eta_{x,y}^{\zeta}(v)) -
      S_{x,y}^{\zeta,v}(\eta_{x,y}^{\zeta}(v'))   \big)
    + \big(  S_{x,y}^{\zeta,v}(\eta_{x,y}^{\zeta}(v')) -
      S_{x,y}^{\zeta,v'}(\eta_{x,y}^{\zeta}(v'))\big)
  \end{equation}
  Recalling the definitions of $v$ and $v'$,
  the second summand in \eqref{eqn:AA2} satisfies
  \begin{equation}
    \big(  S_{x,y}^{\zeta,v}(\eta_{x,y}^{\zeta}(v')) -
      S_{x,y}^{\zeta,v'}(\eta_{x,y}^{\zeta}(v'))\big)
    = - h \eta_{x,y}^{\zeta}(v') \le \Cl{c:etaBd} h,
  \end{equation}
  for some finite constant
  $\Cr{c:etaBd}$, since
  $\eta_{x,y}^{\zeta}(v')\in \triangle$
  for the $x$, $y$, $v'$ under consideration due to
  Lemma~\ref{lem:etazeta}\ref{item:etazetaA}. Moreover, the absolute
  value of the first summand in \eqref{eqn:AA2} can be upper bounded by
  $ch^2/t\ll h$ uniformly for $x$, $y$, $t$ and $h$ under consideration,
  exactly as in the paragraph containing \cite[(3.39)]{DS2022} (this
    proof uses again only estimates that are uniform in $\zeta$). This
  completes the proof of part \ref{timeperturbation}.

  The proof of part \ref{spaceperturbation} follows the lines of the
  proof of the spatial perturbation result, Lemma~4.1 in \cite{CDS2021}.
  Indeed, using the same reasoning as in
  \eqref{eqn:AA0}--\eqref{eqn:AA2}, now choosing $v:=(y-x)/t$ and
  $v' := (y+h-x)/t$, where $x,y,t$ and $h$ are as in the statement and
  $\mathcal T_2$ is assumed to be sufficiently large so that
  $y-x\ge \mathcal N $,  we infer that
  \begin{equation}
    \label{eqn:BB1}
    \frac{E_{x}\Big[e^{\int_0^{t}\xi(X_s)\D s} ;
        X_t \ge y+h \Big]}
    { E_{x}\Big[ e^{\int_0^{t}\xi(X_s)\D s} ;
        X_t \ge y \Big]}
    \le \Cl{c:spatPertUB} \frac{Y^{\approx}_{v'}(x,y+h)}{Y^{\approx}_{v}(x,y)},
  \end{equation}
  as well as
  \begin{equation}
    \label{eqn:BB2}
    \begin{split}
      \ln  \frac{Y^{\approx}_{v'}(x,y+h)}{Y^{\approx}_{v}(x,y)}
      &\le \big(S_{x,y}^{\zeta ,v}(\eta_{x,y}^{\zeta }(v)) -
        S_{x,y+h}^{\zeta ,v'}(\eta_{x,y}^{\zeta }(v))\big)
      \\&\quad+
      \big(S_{x,y+h}^{\zeta ,v'}(\eta_{x,y}^{\zeta }(v))
        -S_{x,y+h}^{\zeta ,v'}(\eta_{x,y+h}^{\zeta }(v'))\big).
    \end{split}
  \end{equation}
  By \eqref{eqn:logMomGen} and \eqref{eqn:Sdef}, the first summand on the
  right-hand side of \eqref{eqn:BB2} (which differs slightly from the
    corresponding one in \cite{CDS2021}, due to item
    \ref{dif:B} above) satisfies
  \begin{equation}
    \begin{split}
      \Big|S_{x,y}^{\zeta ,v}&(\eta_{x,y}^{\zeta }(v)) -
      S_{x,y+h}^{\zeta ,v'}(\eta_{x,y}^{\zeta }(v))\Big|
      \\&= \Big|\ln E_{x} \Big[
        e^{\int_0^{H_{y+h}}(\zeta (X_s)+\eta_{x,y}^\zeta (v))\D s}\Big]
      - \ln E_{x} \Big[
        e^{\int_0^{H_y}(\zeta (X_s)+\eta_{x,y}^\zeta (v) )\D s}\Big]\Big|
      \\&= \Big|\ln E_{y} \Big[
        e^{\int_0^{H_{y+h}}(\zeta (X_s)+\eta_{x,y}^\zeta (v))\D s}\Big]\Big|
      \\&\le  h \sqrt {2(\es - \ei + | \eta_{x,y}^\zeta (v) |)} \le  \Cl{c:penultimate} h,
    \end{split}
  \end{equation}
  where, in the second equality, we applied the strong Markov property at
  time $H_y$, and used  \eqref{eqn:ppp} for the final inequality, and
  $\Cr{c:penultimate} = \Cr{c:penultimate}(V,\triangle, \ei, \es)$ some
  finite constant.

  The second summand on the right-hand side of \eqref{eqn:BB2} is upper
  bounded by $C h^2/t\ll h$ and is thus negligible. This can be proved
  exactly as in \cite[(4.13)--(4.16)]{CDS2021}. Besides
  \cite[Lemma~2.7]{DS2022}, which we already extended in
  Lemma~\ref{lem:etapert}, this proof again only uses uniform estimates
  and thus does not require any modification. This completes the proof of
  the proposition.
\end{proof}

\section{Dependence of solutions to the F-KPP equation on the initial condition}
\label{sec:proofof_contradiction_lemma}

In this section, we prove the key technical lemma,
Lemma~\ref{lem:contradiction} below, which formalises inequalities
\eqref{eqn:tightintrocc} and \eqref{eqn:tightintrodd}, and which provides
the right ordering of two solutions to \eqref{eqn:FKPP} with different
initial conditions. Its proof is based on a careful examination of the
Feynman-Kac representations of these solutions, using all the tools that
were introduced in previous sections.

To state the lemma, we introduce two auxiliary velocities,
\begin{align}
  \label{eqn:definition_of_vone}
  v_1 &:= \sqrt{2 (\es + 1)} \quad \text{and}\\
  \label{eqn:definition_of_vtwo}
  v_2 &: = \inf\{v> v_1+1: |\overline \eta (v)| \ge 2 v_1^2 +2\},
\end{align}
where $\overline \eta (v)$ was defined above \eqref{eqn:lineeta}; note
that display \eqref{eqn:lineeta} also ensures that $v_2$ is finite.
By comparing the BBMRE with the BBM with constant
branching rate $\es$, for which the speed of the maximum is
$\sqrt {2 \es}$, we obtain
\begin{equation*}
  v_0< v_1 < v_2.
\end{equation*}
Recall the notation $w^y$ from below \eqref{eqn:dualityintro}.

\begin{lemma}
  \label{lem:contradiction}
  For each $u>0$ and each $v>v_2$, there exists
  $\Delta_0 = \Delta_0(u,v)>0$ and a $\PP$-a.s.~finite random variable
  $\mathcal{T} = \mathcal  T(u,v)$, such that $\mathbb P$-a.s., for all
  $\Delta>\Delta_0$, $y\in[0,vt]$ and $t\ge \mathcal{T}$,
    \begin{equation}
      \label{eqn:Lemma_contradiction}
      w^{y}(t,y-vt) \ge w^{y+\Delta}(t+u,y-vt).
    \end{equation}
\end{lemma}

\begin{proof}
  We start by upper bounding the right-hand side of
  \eqref{eqn:Lemma_contradiction}. Again by the Feynman-Kac
  representation \eqref{eqn:Feynman-Kac_FKPP} and the fact that
  $\sup_{w \in [0,1]} \widetilde{F}(w) =1$,
  cf.~Proposition~\ref{pro:Feynman-Kac}, it follows for any $\Delta>0$ that
  \begin{align}
    \begin{split}
      \label{eqn:upper_bound_second_term_1}
      w^{y+\Delta}(t+u,y-vt) &=
      E_{y-vt}\Big[e^{\int_0^{t+u}\xi(X_s)\widetilde{F}(w^{y+\Delta}(t+u-s,X_s))\D s}
        ;X_{t+u}\ge y+\Delta \Big] \\
      &\le E_{y-vt}\Big[e^{\int_0^{t+u}\xi(X_s)\D s} ; X_{t+u}\ge y+\Delta \Big].
    \end{split}
  \end{align}
  To the right-hand side of \eqref{eqn:upper_bound_second_term_1} we now
  successively apply both parts of the perturbation result of
  Proposition~\ref{prop:perturbations} (with $V$ sufficiently large, as
    explained before Proposition~\ref{prop:perturbations} and $A=2v$). In
  order to apply them, we let
  $t\ge u\vee \mathcal{T}_1\vee \mathcal{T}_2 =:\mathcal{T}$, where
  $\mathcal{T}_1, \mathcal{T}_2$ are the $\PP$-a.s.~finite random
  variables appearing in the statement of the perturbation lemma. For such
  $t$, we then obtain
  \begin{align}
    \label{eqn:upper_bound_second_term}
    \begin{split}
      w^{y+\Delta}(t+u,y-vt)
      & \le
      \Cr{c:timePert}e^{ \Cr{c:timePert}u}
      E_{y-vt}\Big[e^{\int_0^{t}\xi(X_s)\D s} ; X_{t}\ge y+\Delta \Big]
      \\ &\le \Cr{c:timePert}\Cr{c:spacePert}
      e^{\Cr{c:timePert}u-\Delta/\Cr{c:spacePert}}
      E_{y-vt}\Big[e^{\int_0^{t}\xi(X_s)\D s} ; X_{t}\ge y
        \Big],
    \end{split}
  \end{align}
  which is our first intermediate inequality.

  Let us now turn our focus to bounding the left-hand side of
  \eqref{eqn:Lemma_contradiction} from below. By the Feynman-Kac
  representation \eqref{eqn:Feynman-Kac_FKPP},
  \begin{equation}
    \label{eqn:oiu}
    w^{y}(t,y-vt)   = E_{y-vt} \Big[ \exp\Big\{\int_0^t\xi(X_r)
        \widetilde{F}(w^y(t-r,X_r))\D r\Big\}; X_t\ge y\Big].
  \end{equation}
  We now claim that $\widetilde F$ satisfies,
  \begin{equation}
    \label{eqn:tildeFineq}
    \widetilde{F}(w) = F(w)/w \ge 1-\tfrac{1}{2}(\mu_2-2)w, \qquad
    w\in[0,1].
  \end{equation}
  Indeed, by \eqref{eqn:non-linearity} and the normalisation
  \eqref{eqn:ODprime} of Remark~\ref{rem:muis2}, the non-linearity $F$ of
  \eqref{eqn:FKPP} satisfies $F(0) = 0$ and $F'(0) = 1$. In addition, by
  \eqref{eqn:Fderivatives},   $F''\ge -\mu_2+2$ on $[0,1]$. Therefore, by
  a first order Taylor approximation with Lagrange remainder,
  \begin{equation*}
    F(w) \ge w+\frac{1}{2}\inf\limits_{w^* \in [0,1]}F''(w^*)w^2
    =w-\frac{1}{2}(\mu_2-2)w^2,
  \end{equation*}
  from which \eqref{eqn:tildeFineq} directly follows.

  Plugging \eqref{eqn:tildeFineq} into \eqref{eqn:oiu} and using the
  uniform boundedness \eqref{eqn:uniformly_bounded} from
  Assumption~\ref{ass:environment}, we arrive at
  \begin{equation}
    \label{eqn:lower_bound_first_term}
    w^{y}(t,y-vt)  \ge E_{y-vt}
    \Big[e^{\int_0^{t}\xi(X_s)\D s}\
      e^{ -\frac \es 2(\mu_2-2)
        \int_0^tw^{y}(t-s,X_s)\D s} ; X_{t}\ge y \Big].
  \end{equation}

  In order to obtain a suitable control of the second exponential factor
  in \eqref{eqn:lower_bound_first_term}, we construct an event restricted
  to which the second exponential is bounded from below in a suitable
  way.  For this purpose, recall the definition of $v_1$ from
  \eqref{eqn:definition_of_vone}, and introduce for given $t$, $y$ the
  \emph{moving boundary}
  \begin{equation}
    \beta_{y,t}(s) := y-v_1(t-s), \qquad s\in [0,t].
  \end{equation}
  By $\mathcal{T}_{y,t}:=\inf\{s \ge 0 : X_s = \beta_{y,t}(s)\}$ we
  denote the first hitting time of $\beta_{y,t}$ by a Brownian motion
  started at $y-vt$.

  \begin{figure}
    \centering
    \includegraphics[width=\textwidth]{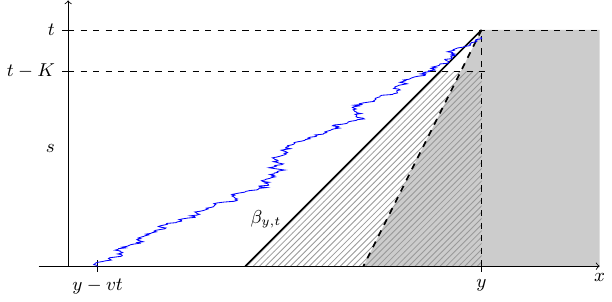}

    \caption{Sketch of a trajectory of the Brownian motion
      $(X_s)_{s\ge 0}$, started at $y-vt$, up until the hitting time $H_y$
      of $y$, which realises the good event~$\mathcal G$. This trajectory
      does not hit the moving barrier $\beta_{y,t}(s)$ (thick solid line)
      in the time interval $[0,t-K]$ and thus avoids the dashed region.
      The function $w^y(t-s,\cdot)$ is close to $1$ in the grey region,
      close to $0$ in its complement, and changes its value from $0$ to
      $1$ in the vicinity of the thick dashed line whose slope is $v_0$.
    }
    \label{figure1}
  \end{figure}

  We claim that for $K>1\vee v_1^{-2}$ to be fixed later, on the good event
  \begin{equation} \label{def:G}
    \mathcal G:= \{\mathcal{T}_{y,t} \in [t-K,t]\}
  \end{equation}
  it holds that
  \begin{equation}
    \label{eqn:RV_upper_bound}
    \int_0^{t-K}w^{y}(t-s,X_s)\D s \le 1,
  \end{equation}
  see Figure~\ref{figure1} for an illustration. Indeed, note that using
  again the Feynman-Kac representation \eqref{eqn:Feynman-Kac_FKPP} as
  well as the uniform boundedness \eqref{eqn:uniformly_bounded} of
  Assumption~\ref{ass:environment}, in combination with the fact that
  $\sup_{w \in [0,1]} \widetilde{F}(w) \le 1$ once more, it holds that
  \begin{align*}
    w^{y}(t-s,X_s)
    &\le E_{X_s}\Big[e^{\int_0^{t-s} \xi(\widetilde{X}_r)\D r };
      \widetilde{X}_{t-s}\ge y \Big]\\
    &\le e^{\es(t-s)}P_{X_s}\big(\widetilde{X}_{t-s}\ge y \big),
  \end{align*}
  where we write $\widetilde{X}$ for an independent Brownian motion
  started at $X_s$ in order to avoid confusion between the two processes.
  On $\mathcal G$ one has that $X_s \le y-v_1(t-s)$ for $s \in [0,t-K]$.
  Hence, by a straightforward coupling argument, on $\mathcal G$ we have
  \begin{equation*}
    P_{X_s}\big(\widetilde{X}_{t-s}\ge y \big)
    \le P_0\big(\widetilde{X}_{t-s}\ge v_1(t-s) \big)
    = P(Z\ge v_1\sqrt{t-s}),
  \end{equation*}
  where $Z$ is a standard Gaussian random variable. Using this in
  combination with a standard Gaussian bound (see
    e.g.~\cite[(1.2.2)]{AdTa-07}) and taking advantage of the fact that
  $v_1\sqrt{(t-s)}\ge v_1\sqrt K \ge 1$, it follows that on $\mathcal G$
  we can upper bound
  \begin{align}
    \label{eqn:w_bound_moving_boundary}
    \begin{split}
      \int_0^{t-K}&w^{y}(t-s,X_s)\D s
      \le \int_0^{t-K}e^{\es(t-s)}P\big(Z\ge v_1\sqrt{t-s} \big)\D s\\
      &\le \frac{1}{\sqrt{2\pi }}\int_0^{t-K}e^{-(v_1^2/2-\es)(t-s)} \D s
      =\frac{1}{\sqrt{2\pi }} \int_K^{t}e^{-(v_1^2/2-\es)z} \D z\\
      &\le \frac{1}{\sqrt{2\pi }(v_1^2/2-\es)}e^{-K(v_1^2/2-\es)} \le 1,
    \end{split}
  \end{align}
  where in the last inequality we used $v_1^2/2-\es = 1$, which holds by
  \eqref{eqn:definition_of_vone}. This proves \eqref{eqn:RV_upper_bound}.

  Coming back to the task of finding a lower bound for the right-hand
  side of \eqref{eqn:lower_bound_first_term}, we infer from the above
  discussion that on $\mathcal G$ we can use \eqref{eqn:RV_upper_bound}
  to bound the second exponential factor  on the right-hand side of
  \eqref{eqn:lower_bound_first_term} by
  \begin{align}
    \label{eqn:lower_bound_last}
    \begin{split}
      e^{-\frac{\es}{2}(\mu_2-2) \int_0^{t}w^{y}(t-s,X_s)\D s}
      &\ge e^{ -\frac{\es}{2}(\mu_2-2)
        \big(1+ \int_{t-K}^tw^{y}(t-s,X_s)\D s\big)}\\
      &\ge e^{ -\frac{\es}{2}(\mu_2-2) (1+K)},
    \end{split}
  \end{align}
  where in the last inequality we used that $0\le w^{y}(s,y)\le 1$
  uniformly for all $(s,y) \in [0,\infty)\times \R$.
  Consequently, by restricting the expectation on the right-hand side of
  \eqref{eqn:lower_bound_first_term} to $\mathcal G$, it follows by
  \eqref{eqn:lower_bound_last} that whenever $v>v_1$, then
  \begin{equation}
    \label{eqn:lower_bound_first_term_2}
    w^{y}(t,y-vt) \ge
    e^{ -\frac{\es}{2}(\mu_2-2)(1+K)}
    E_{y-vt}\Big[e^{\int_0^{t}\xi(X_s)\D s} ; X_{t}\ge y ,\ \mathcal G\Big].
  \end{equation}
  This is our second intermediate inequality.

  In order to finish the proof of \eqref{eqn:Lemma_contradiction}, we
  need to compare the expectations on the right-hand side of
  \eqref{eqn:upper_bound_second_term} and on the right-hand side of
  \eqref{eqn:lower_bound_first_term_2}. This is the purpose of the
  following lemma.

  \begin{lemma}
    \label{lem:Claim}
    Let $v_2$ be as in \eqref{eqn:definition_of_vtwo} and $\mathcal G$ as
    in \eqref{def:G}. Then, for every $v>v_2$ there exist constants
    $K=K(v)$ and $\widetilde{C} = \widetilde{C}(v) \in (0,\infty)$ such
    that $\mathbb P$-a.s., for all $t$ large enough and all $y \in [0,vt]$,
    one has
    \begin{equation}
      \label{eqn:Claim}
      E_{y-vt}\big[ e^{\int_0^t\xi(X_s)\D s} ;
        X_t \ge y\big]
      \le \widetilde{C} E_{y-vt}\big[ e^{\int_0^t\xi(X_s)\D s };
        X_t \ge y,\, \mathcal G\big].
    \end{equation}
  \end{lemma}

  We postpone the proof of Lemma~\ref{lem:Claim} and complete the proof
  of Lemma~\ref{lem:contradiction} first. By combining  the lower bound
  \eqref{eqn:lower_bound_first_term_2}, the upper bound
  \eqref{eqn:upper_bound_second_term} and Lemma~\ref{lem:Claim}, we obtain.
  \begin{align*}
    &w^{y}(t,y-vt)-w^{y+\Delta}(t+u,y-vt)
    \\&\ge  \big(e^{ -\frac{\es}{2}(\mu_2-2)(K+1)}
      -\widetilde{C}\Cr{c:timePert}\Cr{c:spacePert}
      e^{\Cr{c:timePert}u-\Delta/\Cr{c:timePert}}\big)
    E_{y-vt}\Big[e^{\int_0^{t}\xi(X_s)\D s}
      ; X_{t}\ge y, \mathcal G \Big].
  \end{align*}
  For every $\Delta$ satisfying
  \begin{equation*}
    \Delta \ge \Delta_0
    := \Cr{c:spacePert}\Big( \Cr{c:timePert} u+ \frac{\es}{2}(\mu_2-2) (K+1)
      + \ln (\widetilde{C}\Cr{c:timePert}\Cr{c:spacePert}) \Big),
  \end{equation*}
  the right-hand side is positive, which proves
  \eqref{eqn:Lemma_contradiction} and thus the lemma.
\end{proof}

\begin{proof}[Proof of Lemma~\ref{lem:Claim}]
  To prove the lemma, we use the machinery of tilted measures as
  introduced in Section~\ref{sec:tilted_measure}. We recall the notation
  $\zeta = \xi - \es$ from \eqref{eqn:definition_of_zeta} and observe
  that, by multiplying both sides of \eqref{eqn:Claim} by $e^{-\es\, t}$,
  it is sufficient to show \eqref{eqn:Claim} with $\zeta $ in place of
  $\xi$.

  We start by proving an upper bound for the left-hand side of
  \eqref{eqn:Claim} in terms of tilted measures. By
  Lemma~\ref{lem:YaproxFK} there exist  constants $C,L<\infty$ such that
  for any $\eta <0$, for $t$ large enough uniformly in $y\in [0,vt]$ it
  holds that
  \begin{align}
    \begin{split}
      \label{eqn:step_1}
      E_{y-vt}\big[
        e^{\int_0^t\zeta(X_s)\D s} ; X_t \ge y\big]
      &\le C E_{y-vt}\big[e^{\int_0^{H_y}\zeta(X_s)\D s}\ ; H_y\in [ t-L,t] \big]\\
      &\le Ce^{-\eta t} Z_{y-vt,y}^{\zeta,\eta}
      P_{y-vt}^{\zeta,\eta}\big(H_y \in [t-L,t]\big).
    \end{split}
  \end{align}

  In the next step, we bound the expression appearing on the right-hand
  side of \eqref{eqn:Claim} from below. To this end, let
  $p_y^{\zeta,\eta}(t) := P_y^{\zeta,\eta}(X_t\ge y)$. Using the strong
  Markov property we obtain
  \begin{align}
    \begin{split}
      \label{eqn:step_2}
      &E_{y-vt}\big[ e^{\int_0^t\zeta(X_s)\D s }
        ; X_t \ge y, \mathcal{T}_{y,t} \ge t-K\big]
      \\ &\ge
      e^{-(\es-\ei)K}
      E_{y-vt}\big[ e^{\int_0^{H_y}\zeta(X_s)\D s }
        ;  H_y\in [t-K,t], X_t \ge y,  \mathcal{T}_{y,t} \ge t-K\big]
      \\ &\ge
      e^{-(\es-\ei-\eta)K}e^{-\eta t}
      E_{y-vt}\big[ e^{\int_0^{H_y}(\zeta(X_s)+\eta) \D s }
        ;H_y\in [t-K,t], X_t \ge y,   \mathcal{T}_{y,t} \ge t-K \big]
      \\& =
      e^{-(\es-\ei-\eta)K}e^{-\eta t} Z_{y-vt,y}^{\zeta,\eta}
      E_{y-vt}^{\zeta,\eta}\big[
        p_y^{\zeta,\eta}(t-H_y),   H_y \in [t-K,t],\mathcal{T}_{y,t} \ge t-K
        \big]
      \\& \ge
      \frac{1}{2}e^{-(\es-\ei-\eta)K}e^{-\eta t} Z_{y-vt,y}^{\zeta,\eta}
      P_{y-vt}^{\zeta,\eta}\big( H_y \in [t-K,t],\mathcal{T}_{y,t} \ge t-K \big),
    \end{split}
  \end{align}
  where in the last inequality we used Lemma~\ref{lem:comparison_lemma}
  to infer that for any $\eta < 0$ and $s\ge 0$ one has
  $p_y^{\zeta,\eta}(s)\ge P_0^{\sqrt{2|\eta|}}(X_s\ge 0) \ge 1/2$.

  In view of \eqref{eqn:step_1} and \eqref{eqn:step_2}, in order to
  complete the proof of Lemma~\ref{lem:Claim}, it is sufficient to show that
  \begin{equation}
    \label{eqn:reductin}
    P_{y-vt}^{\zeta,\eta}\big(H_y \in [t-L,t]\big)
    \le C P_{y-vt}^{\zeta,\eta}\big(
      H_y \in [t-K,t],\mathcal{T}_{y,t} \ge t-K \big),
  \end{equation}
  for some suitably chosen parameter $\eta $ and constants $C,K,L$,
  $\mathbb P$-a.s.~for all $t$ large, uniformly in $y\in [0,vt]$.

  To this end, we will need two further auxiliary lemmas. The first one
  will be used to upper bound the probability appearing on the right-hand
  side of \eqref{eqn:reductin}, and also specifies the range of suitable
  $\eta $'s.

  \begin{lemma}
    \label{lem:Step_3}
    Let $\eta<0$ be such that $\sqrt{2|\eta|}> v_1(1+\frac{2L}{K})$, and
    let $0<L<K$ be such that $L/K\le 1/3$. Then, $\mathbb P$-a.s.~for
    every $y\in \mathbb R$ and $v>v_1$,
    \begin{equation}
      \label{eqn:Step_3}
      P_{y-vt}^{\zeta,\eta}\big( H_y\le t, \mathcal{T}_{y,t}\le t-K\big)
      \le 2 P_{y-vt}^{\zeta,\eta}( H_y< t-L).
    \end{equation}
  \end{lemma}

  The second auxiliary lemma is a quantitative extension of a part of
  Proposition 3.5 of \cite{DS2022}. It states that under the tilted
  measure, if the tilting is not too strong,  the probability of crossing a
  large interval in time $t$ is comparable to the probability of crossing
  the same interval in time $t-L$.

  \begin{lemma}
    \label{lem:Step_4}
    For every $v>v_c$ there is $c = c(v)<\infty$ such that
    for all $L$ large enough and
    $\eta \in (\overline \eta (v)+\frac cL,0)$, $\mathbb P$-a.s.~for all
    $t$ large enough and $|y| \le 2vt$,
    \begin{equation*}
      P_{y-vt}^{\zeta,\eta}(H_y \le  t-L)
      \le \frac{1}{4}P_{y-vt}^{\zeta,\eta}(H_y \le t),
    \end{equation*}
    and as a consequence,
    \begin{equation*}
      P_{y-vt}^{\zeta,\eta}(H_y \le t-L)
      \le \frac{1}{3}P_{y-vt}^{\zeta,\eta}\big(H_y \in (t-L,t] \big).
    \end{equation*}
  \end{lemma}

  We postpone the proofs of these two lemmas to the end of the current
  section. We now come back to the proof of Lemma~\ref{lem:Claim} and
  complete it by showing \eqref{eqn:reductin}. To this end, we choose the
  parameters $\eta$, $K$, and $L$ in such a way that the previous two
  lemmas can be used simultaneously. More precisely, for a given
  $v\ge v_2$ we fix arbitrary $\eta $ so that
  \begin{equation}
    \label{eqn:fixeta}
    |\overline \eta (v)| -1 > |\eta | > 2 v_1^2,
  \end{equation}
  which is possible by the definition of $v_2$ in
  \eqref{eqn:definition_of_vtwo}. Then we fix $L$ as large as required in
  Lemma~\ref{lem:Step_4}. Consequently, due to \eqref{eqn:fixeta}, the
  required assumptions on $\eta$ are satisfied in our setting. Finally,
  we fix $K\ge 3 L$ and observe that, in combination with
  \eqref{eqn:fixeta},
  $\sqrt {2|\eta |} > 2v_1 \ge v_1(1+\frac {2L}{K})$, so that the
  assumptions of Lemma~\ref{lem:Step_3} are satisfied also.

  With this choice of constants, noting that
  $\{H_y \in [t-K,t],\mathcal{T}_{y,t} \ge t-K\}
  = \{H_y \le t,\mathcal{T}_{y,t} \ge t-K\}$
  (cf.\ Figure \ref{figure1} also), the right-hand side of
  \eqref{eqn:reductin} satisfies
  \begin{equation}
    \begin{split}
      \label{eqn:aa}
      P_{y-vt}^{\zeta,\eta}&\big( H_y \in [t-K,t],\mathcal{T}_{y,t} \ge t-K \big)
      \\&=
      P_{y-vt}^{\zeta,\eta}\big( H_y \le t\big)
      - P_{y-vt}^{\zeta,\eta}\big( H_y\le t, \mathcal{T}_{y,t} < t-K \big)
      \\&\ge
      P_{y-vt}^{\zeta,\eta}\big( H_y \le t\big)
      -2 P_{y-vt}^{\zeta,\eta}\big( H_y \le t-L\big),
    \end{split}
  \end{equation}
  where the last inequality follows from Lemma~\ref{lem:Step_3}. This can
  be written as
  \begin{equation}
    \label{eqn:ab}
    P_{y-vt}^{\zeta,\eta}\big( H_y \in [t-L,t]\big)
    - P_{y-vt}^{\zeta,\eta}\big( H_y \le t-L\big) \ge
    \frac 23 P_{y-vt}^{\zeta,\eta}\big( H_y \in [t-L,t]\big),
  \end{equation}
  where the last inequality is a direct consequence of
  Lemma~\ref{lem:Step_4}. Now combining  \eqref{eqn:aa} and
  \eqref{eqn:ab} we obtain \eqref{eqn:reductin}, which completes the proof.
\end{proof}

It remains to provide the proofs of Lemmas~\ref{lem:Step_3} and
\ref{lem:Step_4}.

\begin{proof}[Proof of Lemma \ref{lem:Step_3}]
  Using the tower property for conditional expectations, we obtain
  \begin{align}
    \label{eqn:sigma_past}
    \begin{split}
      P_{y-vt}^{\zeta,\eta}(H_y<t-L)
      &\ge  P_{y-vt}^{\zeta,\eta}(H_y<t-L, \mathcal{T}_{y,t}\le t-K)\\
      & = E_{y-vt}^{\zeta,\eta}\big[
        \bbone_{\{\mathcal{T}_{y,t}\le t-K\}}
        P_{y-vt}^{\zeta,\eta}( H_y<t-L\mid {\mathcal F_{\mathcal{T}_{y,t}}})
        \big],
    \end{split}
  \end{align}
  where $\mathcal F_{\mathcal{T}_{y,t}}$ is the canonical stopped $\sigma$-algebra
  associated to $\mathcal{T}_{y,t}$. It follows from
  Lemma~\ref{lem:comparison_lemma} that the drift of $X$ under the tilted
  measure $P_{y-vt}^{\zeta,\eta}$ is always larger than $\sqrt{2|\eta|}$.
  On the event $\{0\le\mathcal{T}_{y,t}\le t-K\}$, by the strong Markov
  property at time $\mathcal T_{y,t}$ and using that
  $X_{\mathcal T_{y,t}}=\beta_{y,t}(\mathcal T_{y,t})$, it holds that
  \begin{align} \label{eqn:probEst}
    \begin{split}
      P_{y-vt}^{\zeta,\eta}\big (H_y< t-L \mid \mathcal F_{\mathcal{T}_{y,t}} \big)
      &= P_{X_{\mathcal T_{y,t}}}^{\zeta,\eta}\big (H_y< t-L-\mathcal T_{y,t}\big)
      \\&\ge \inf_{0\le u \le t-K} P_{\beta_{y,t}(u)}^{\zeta,\eta}(H_y\le t-u-L )
      \\&\ge \inf_{0 \le u \le t-K} P_{\beta_{y,t}(u)}^{\sqrt{2|\eta|}}(H_y\le
        t- u-L ).
    \end{split}
  \end{align}
  Recalling the assumptions of the lemma,  for $u\in[0,t-K]$ we have that
  \begin{equation}
    \label{eqn:drifteqn}
    \begin{split}
      E_{\beta_{y,t}(u)}^{\sqrt{2|\eta|}}(X_{t-u-L})
      &= \beta_{y,t}(u) + \sqrt{2|\eta|}(t-u-L)
      \\&\ge y-v_1(t-u) + v_1\big(1+\tfrac {2L}K\big)(t-u-L)
      \\&\ge  y- v_1 L + v_1\tfrac {2L}K(K-L) \ge y + \tfrac 13 v_1 L \ge
      y,
    \end{split}
  \end{equation}
  where for the penultimate inequality we used $K-L\ge \frac 23 K$, which
  holds by assumption. In combination with the fact that $X$ is a
  Brownian motion with drift under $P_{\beta_{y,t}(u)}^{\sqrt{2|\eta|}}$,
  it follows that the probability on the right-hand side of
  \eqref{eqn:probEst} is at least $1/2$. Plugging this back into
  \eqref{eqn:sigma_past} we arrive at
  \begin{align*}
    P_{y-vt}^{\zeta,\eta}(H_y < t-L)
    &\ge \frac{1}{2} P_{y-vt}^{\zeta,\eta}
    \big(\mathcal{T}_{y,t}\le t-K\big)
    \\&\ge \frac{1}{2} P_{y-vt}^{\zeta,\eta}
    \big(\mathcal{T}_{y,t}\le t-K, H_y \le t \big),
  \end{align*}
  as claimed.
\end{proof}

Next, we give the proof of Lemma~\ref{lem:Step_4}.

\begin{proof}[Proof of Lemma~\ref{lem:Step_4}]
  The first part of the proof of this lemma follows the same steps as the
  proof of Proposition~3.5 of \cite{DS2022} (see also the proof of
    Lemma~\ref{lem:Ycomparability}). By
  Lemma~\ref{lem:etazeta}\ref{item:etazetaA}, $\mathbb P$-a.s.~for all $t$
  large enough, and all $|y|\le 2vt$, there exist constants
  $\eta_{y-vt,y}^{\zeta }(v)$ so that
  \begin{equation}
    E_{y-vt}^{\zeta, \eta_{y-vt,y}^{\zeta }(v)}[H_y]=t.
  \end{equation}
  To simplify the notation we write
  $\widetilde \eta  := \eta_{y-vt,y}^{\zeta }(v)$. Using
  Lemma~\ref{lem:etazeta}\ref{item:etazetaB}, we can assume that
  $\widetilde \eta  < \overline \eta (v) + \frac c{2L}$, and thus, by the
  hypothesis of the lemma,
  \begin{equation}
    \eta - \widetilde \eta > \frac c{2L}.
  \end{equation}
  By definition of the tilted measures, cf.
  \eqref{eqn:definition_of_tilted_measure},
  \begin{equation}
    \label{eqn:arrival_before}
    \begin{split}
      P_{y-vt}^{\zeta,\eta}&(H_y \le t-L) = \frac{1}{Z^{\zeta,\eta }_{y-vt,y}}\,
      E_{y-vt}\Big[e^{\int_0^{H_y}(\zeta (X_s)+\eta )\D s}; H_y\le t-L\Big]
      \\&=
      \frac{Z_{y-vt,y}^{\zeta,\widetilde \eta} }{Z_{y-vt,y}^{\zeta,\eta}}\,
      \frac{1}{Z_{y-vt,y}^{\zeta,\widetilde \eta }}\,
      E_{y-vt}\Big[e^{\int_0^{H_y}(\zeta(X_s)+\widetilde \eta )\D s}
        e^{-H_y(\widetilde\eta-\eta)}; H_y \le t-L\Big].
      \\&=
      \frac{Z_{y-vt,y}^{\zeta,\widetilde \eta} }{Z_{y-vt,y}^{\zeta,\eta}}\,
      E_{y-vt}^{\zeta ,\widetilde\eta }\Big[
        e^{-H_y(\widetilde\eta-\eta)}; H_y \le t-L\Big].
    \end{split}
  \end{equation}
  Define random variables
  $\tau_i = H_{y-vt+i}-H_{y-vt+i-1}$, $i = 1,\dots,\lfloor vt \rfloor$, and
  $\tau_{vt} = H_y-H_{y-vt+\lfloor vt\rfloor}$, so that
  $\sum_{i = 1}^{\lfloor vt \rfloor} \tau_i + \tau_{vt} = H_y$,
  and their centred versions
  $\widehat\tau_i = \tau_i-E_{y-vt}^{\zeta ,\widetilde \eta }[\tau_i]$ for
  $i = 1,\dots, \lfloor vt\rfloor$, and
  $\widehat\tau_{vt} = \tau_{vt}-E_{y-vt}^{\zeta ,\widetilde\eta}[\tau_{vt}]$.
  Further, let
  \begin{equation}
    Y_{y-vt,y}^{\zeta} := \tfrac{(\widetilde\eta-\eta)}
    {\widetilde \sigma }\big(\sum_{i = 1}^{\lfloor vt \rfloor} \widehat\tau_i +
      \widehat\tau_{vt}\big),
  \end{equation}
  where
  \begin{equation}
    \label{eqn:definition_sigma}
    \widetilde \sigma  = \widetilde\sigma_{y-vt,y}^{\zeta}(v)
    = |\widetilde \eta -\eta|\sqrt{\var_{P_{y-vt}^{\zeta,\widetilde\eta}}(H_y)}.
  \end{equation}
  is chosen so that the variance of $Y_{y-vt,y}^\zeta $ is one. Denoting
  by $\mu_{y-vt,y}^\zeta $ the distribution of $Y_{y-vt,y}^\zeta $ under
  $P_{y-vt,y}^{\zeta ,\widetilde \eta }$, using
  also the fact that $E_{y-vt,y}^{\zeta, \widetilde \eta }[H_y] = t$, by the definition
  of $\widetilde \eta $, \eqref{eqn:arrival_before} can be rewritten as
  \begin{align}
    \begin{split}
      \label{eqn:bulk_arrival_integral}
      P_{y-vt}^{\zeta,\widetilde\eta}&(H_y \le t-L)
      \\&=\frac{Z_{y-vt,y}^{\zeta,\widetilde\eta}}{Z_{y-vt,y}^{\zeta,\eta}}
      \,e^{(\eta-\widetilde\eta)t}
      E_{y-vt}^{\zeta, \widetilde\eta}
      \Big[ e^{-\widetilde\sigma Y_{y-vt,y}^{\zeta}} ; Y_{y-vt,y}^{\zeta} \in
        \Big[\frac{L(\eta-\widetilde\eta)}{\widetilde\sigma},\infty\Big)\Big]
      \\
      &=\frac{Z_{y-vt,y}^{\zeta,\widetilde\eta}}{Z_{y-vt,y}^{\zeta,\eta}}
      \,e^{(\eta-\widetilde\eta)t}
      \int_{L(\eta-\widetilde\eta)/\widetilde\sigma}^{\infty} e^{-\widetilde \sigma u}
      \mu_{y-vt,y}^{\zeta}(\D u).
    \end{split}
  \end{align}
  Setting $L=0$ in the above formula we further obtain
  \begin{equation}
    \label{eqn:arrival_before_t_integral}
    P_{y-vt}^{\zeta,\eta}(H_y \le t)
    =\frac{Z_{y-vt,y}^{\zeta,\widetilde\eta}}{Z_{y-vt,y}^{\zeta,\eta}}
    \,e^{(\eta-\widetilde\eta)t}
    \int_{0}^{\infty} e^{-\widetilde \sigma u}
    \mu_{y-vt,y}^{\zeta}(\D u),
  \end{equation}
  Hence, to finish the proof of the lemma, it suffices to show that the
  integral on the right-hand side of \eqref{eqn:bulk_arrival_integral} is
  at most $1/4$ of the integral on the right-hand side
  of~\eqref{eqn:arrival_before_t_integral}.

  To see this we proceed as in the proof of Lemma~3.6 of \cite{DS2022}.
  By the strong Markov property the random variables $\widehat\tau_i$,
  $i = 1,\dots, \lfloor vt \rfloor$, and $\widehat \tau_{vt}$ are
  independent under $P_{y-vt}^{\zeta,\widetilde\eta}$. Further, it is a
  straightforward consequence of the definitions of the logarithmic
  moment generating functions in \cite[(2.7)]{DS2022} and their being
  well defined for $\eta<0$ that these random variables have uniform
  exponential moments. Moreover, recall that $\widetilde\sigma$ was
  chosen such that the variance of $\mu_{y-vt,y}^{\zeta}$ is one. This
  allows the application of a local central limit theorem for normalised
  sums of independent random variables \cite[Theorem
    13.3]{Bhattacharya1986nonrmal}, which implies that
  \begin{equation}
    \label{eqn:application_lclt}
    \sup_{B}|\mu_{y-vt,y}^{\zeta}(B)-\Phi(B)|\le \Cl{c:LCLT} (\lceil vt \rceil)^{-1/2},
  \end{equation}
  where the supremum is taken over all intervals $B$ in $\R$ and $\Phi$
  denotes the standard Gaussian measure. Note that the constant
  $\Cr{c:LCLT}$ in the last display depends only on the uniform bound of
  the exponential moments of the $\widehat\tau_i$s. Without loss of
  generality, we can assume that $\Cr{c:LCLT}>4$.  We also note that by
  \cite[(3.8)]{DS2022} (see also \eqref{eqn:sigmaas}) the variance
  $\widetilde\sigma^2$ defined in \eqref{eqn:definition_sigma} satisfies
  for $\PP$-a.e.~$\zeta$ and $t$ large enough
  \begin{equation}
    \label{eqn:sigma_bounds}
    \Cl{c:var}^{-1}\sqrt{\lceil vt \rceil}
    \le \widetilde\sigma\le \Cr{c:var} \sqrt{\lceil vt \rceil}.
  \end{equation}

  We now have all ingredients to finish the proof. To this end, we assume
  that the constant $c$ from the statement of the lemma satisfies the
  inequality
  \begin{equation}
    \label{eqn:ell}
    \ell := \frac{L(\eta -\widetilde \eta )}{\widetilde \sigma }
    \ge \frac c{2\Cr{c:var}  \sqrt{vt}} \ge \frac{20 c_1}{\sqrt {vt}}.
  \end{equation}
  To bound the integral in \eqref{eqn:bulk_arrival_integral} from above,
  we observe that for any interval $(a,b)$ of length $\ell$ we have
  $\Phi((a,b))\le \ell/\sqrt{2\pi }$ and thus
  $\mu_{y-vt,y}^{\zeta}((a,b))\le (\ell + \Cr{c:LCLT}/\sqrt{vt}) \le 2 \ell $,
  by \eqref{eqn:ell}. Therefore, using \eqref{eqn:sigma_bounds} in the
  last step,
  \begin{equation}
    \label{eqn:intup}
    \begin{split}
      \int_{L(\eta-\widetilde\eta)/\widetilde\sigma}^{\infty} e^{-\widetilde \sigma u}
      \mu_{y-vt,y}^{\zeta}(\D u)
      &\le \sum_{i=1}^\infty e^{-\widetilde \sigma i \ell}
      \mu_{y-vt,y}^{\zeta}((i\ell, (i+1)\ell))
      \\&\le\frac{ 2\ell e^{-\widetilde\sigma  \ell }}{1-e^{-\widetilde\sigma \ell}}
      \le\frac{ 2\widetilde \sigma \ell e^{-\widetilde\sigma  \ell }}{1-e^{-\widetilde\sigma \ell}}
      \cdot \frac{\Cr{c:var}}{\sqrt{vt}}.
    \end{split}
  \end{equation}

  On the other hand, using the rough bound $\Phi((0,x))\ge x/5$ which holds
  for small enough $x$, and \eqref{eqn:ell},
  \begin{equation}
    \label{eqn:intdown}
    \begin{split}
      \int_{0}^{\infty} e^{-\widetilde \sigma u}
      &\mu_{y-vt,y}^{\zeta}(\D u)
      \ge
      \int_{0}^{L(\eta-\widetilde\eta)/2\widetilde\sigma} e^{-\widetilde \sigma u}
      \mu_{y-vt,y}^{\zeta}(\D u)
      \\&\ge e^{-\widetilde \sigma \ell /2} \mu_{y-vt,y}^{\zeta}((0,\ell/2))
      \ge e^{-\widetilde \sigma \ell /2}
      \Big(\Phi((0,\ell/2))-\frac {\Cr{c:LCLT}}{\sqrt {vt}}\Big)
      \\&\ge e^{-\widetilde \sigma \ell /2} \frac{c_1}{\sqrt {vt}}.
    \end{split}
  \end{equation}
  By increasing the value of the constant $c$ and thus of
  $\widetilde\sigma \ell \ge c/2$, the right-hand side of \eqref{eqn:intup}
  can be made at most $1/4$ as large as the right-hand side of
  \eqref{eqn:intdown}. This completes the proof of the lemma.
\end{proof}

\section{Proof of the tightness of the maximum of the BBMRE}
\label{sec:proofof_tightness}

We are now ready to prove the main theorem of this paper.

\begin{proof}[Proof of Theorem~\ref{thm:BBMRE_is_tight}]
  For $\varepsilon \in (0,1/2)$ let $x_t  = x_t(\varepsilon)\in \R$ the
  unique location where
  \begin{equation}
     \label{eqn:xt}
      w^{x_t}(t,0) = \PPP_0^\xi(M(t) \ge x_t) = \varepsilon.
  \end{equation}
  As already
  explained in Section~\ref{ssec:strategy}, to show the tightness of the
  re-centred maximum $M(t)$ we need to show that there
  exists $\Delta = \Delta(\varepsilon)<\infty$ such that for all $t>0$ it
  holds that
  \begin{equation}
    \label{eqn:goal_tight}
    w^{x_t-\Delta}(t,0)=\PPP_0^{\xi}(M(t)\ge x_t-\Delta) >1-\varepsilon.
  \end{equation}

  Note that \eqref{eqn:xt} and the law of large numbers for $M(t)$ (that is
    $\lim_{t \to \infty} M(t)/t = v_0$, cf. \eqref{eqn:MtLLN}) imply that
  \begin{equation}
    \label{eqn:LLN_x_t}
    \lim_{t \to \infty} \frac{x_t}{t} = v_0, \quad\mathbb P\text{-a.s.}
  \end{equation}
  A direct consequence of \eqref{eqn:LLN_x_t} is that for $t$ large
  enough, we can guarantee that $x_t$ is large enough in order to apply
  Corollary~\ref{cor:almost_surely_finite_time} to $w^{x_t}(t,0)$.
  Therefore, for large enough $t$, we infer the existence of a $\PP$-a.s.\
  finite random time $T<\infty$ such that
  \begin{equation}
    \label{eqn:contradiction_conclusion_1}
    w^{x_t}(t+t',0) \ge 1-\varepsilon \quad
    \text{for all $t'\in [T,T+1]$}.
  \end{equation}

  For any $u \in \N$ define the subset $\Omega_u = \{ T \in [u-1,u)\}$ of
  the probability space on which $\xi$ is defined. We now consider
  $\xi \in \Omega_u$. Observe that \eqref{eqn:goal_tight} would follow
  from \eqref{eqn:contradiction_conclusion_1} on $\Omega_u$, if for a
  suitably large $\Delta<\infty$ we had
  \begin{equation}
    \label{eqn:right_ordering}
    w^{x_t-\Delta}(t,0)\ge w^{x_t}(t+u,0).
  \end{equation}

  Instead of comparing these two functions directly at $x=0$ we use the
  Sturmian principle, to relate the inequality \eqref{eqn:right_ordering}
  at the origin to an inequality at some point on the negative half-line.
  More precisely, recall from Section~\ref{ssec:sturmian_principle} that
  for any $t>0$, $u>0$ and $\widetilde\Delta<\infty$ the difference
  \begin{equation*}
    W^{u,\widetilde \Delta}(t,x):=w^{x_t-\widetilde \Delta}(t,x)-w^{x_t}(t+u,x)
  \end{equation*}
  solves a linear parabolic equation of the form
  \eqref{eqn:linear_parabolic_equation}, with initial condition
  \begin{equation*}
    W^{u,\widetilde \Delta}(0,x) = \bbone_{[x_t-\widetilde \Delta,\infty)}(x)-w^{x_t}(u,x).
  \end{equation*}

  Since $0<w^{x_t}(u,x)<1$ for all $u>0$ and $x \in \R$, cf.\
  \eqref{eqn:sol01bd}, it follows moreover, that
  $W^{u,\widetilde \Delta}(0,x)>0$ for $x>x_t-\widetilde\Delta$ and
  $W^{u,\widetilde \Delta}(0,x)<0$ for $x<x_t-\widetilde\Delta$.
  Therefore, it holds by Lemma~\ref{lem:sturmian_principle} that for all
  $t>0$ the sets
  \begin{equation*}
    \{ x \in \R : w^{x_t-\widetilde \Delta}(t,x)> w^{x_t}(t+u,x)\}
    = \{x \in \R : W^{u,\widetilde \Delta}(t,x)>0\}
  \end{equation*}
  are open intervals unbounded to the right.
  Thus, in order to prove \eqref{eqn:right_ordering} it suffices to find
  some $x^* = x^*(t)<0$ and $\Delta<\infty$ such that $W^{u,\Delta}(t,x^*)>0$,
  as this implies $0 \in \{x \in \R : W^{u,\Delta }(t,x)>0\}$, which in turn
  implies \eqref{eqn:right_ordering}; for an illustration of this
  argument see Figure~\ref{figure:ordering}.

  \begin{figure}
    \centering
    \includegraphics[width=\textwidth]{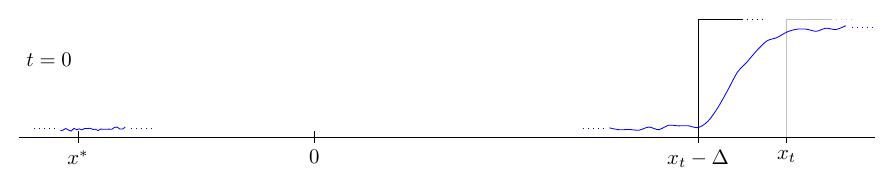}

    \includegraphics[width=\textwidth]{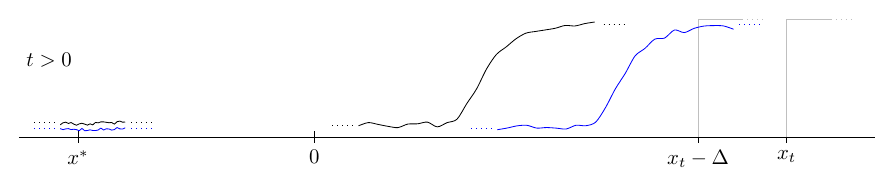}

    \caption{The top figure shows the graph of the function
      $w^{x_t-\Delta}(0,\cdot) = \bbone_{[x_t-\Delta,\infty)}(\cdot)$ in
      black and the function $w^{x_t}(T,\cdot)$ in blue. The lower figure
      shows the graph of the same functions at some positive time $t>0$.
      By the Sturmian principle, for any $t>0$, the region where
      $w^{x_t-\Delta}(t,\cdot)$ dominates $w^{x_t}(t+T,\cdot)$ is an
      interval that contains $[x^*,\infty)$.}
    \label{figure:ordering}
  \end{figure}

  To find such $x^*(t)$ take any $v>v_2$, where $v_2$ is defined in
  \eqref{eqn:definition_of_vtwo}. Since $v_2>v_0$ and
  $(x_t-\Delta)/t\to v_0$, by \eqref{eqn:LLN_x_t}, it follows that
  $x_t-\Delta \in [0,vt]$ for all $t$ that are sufficiently large.
  Consequently we can apply Lemma~\ref{lem:contradiction} and infer the
  existence of an $\PP$-a.s.~finite random variable  $\mathcal{T}(u,v)$ and
  some $\Delta_0(u,v)>0$ such that if we require, additional to the
  previous conditions on the size of $t$, that $t> \mathcal{T}(u,v)$ and
  that $\Delta>\Delta_0(u,v)$, then
  \begin{equation*}
    w^{x_t-\Delta}(t,x^*) \ge w^{x_t}(t+u,x^*),
  \end{equation*}
  with $x^* = x_t-\Delta-vt<0$.
  By the previous discussion this implies
  \begin{equation*}
     w^{x_t-\Delta}(t,0) \ge w^{x_t}(t+u,0) \ge 1-\varepsilon,
  \end{equation*}
  and hence tightness of the family
  $(M(t)-m^\xi(t))_{t\ge 0}$  for $\PP$-a.e.~$\xi\in \Omega_u$.
  As $\PP(\Omega) = \PP(\cup_{u\ge 1} \Omega_u) = 1$, this completes the proof.
\end{proof}

\begin{funding}
  The first and the third author were supported by Deutsche
  Forschungsgemeinschaft through DFG project no.\ 443869423 and by
  Swiss National Science Foundation through SNF project no.\ 200021E\_193063
  in the context of a joint DFG-SNF project within in the framework of
  DFG priority programme SPP 2265 Random Geometric Systems.
\end{funding}

\bibliographystyle{imsart-number}
\bibliography{tightness}

\providecommand{\arXivhref}[1]{{\href{https://arxiv.org/abs/#1}{\emph{arXiv:#1}}}} \providecommand{\arXiv}[1]{{Preprint, available at \arXivhref{#1}}}
\begin{thebibliography}{47}

\bibitem{AddarioBerry2009minima}
\begin{barticle}[author]
\bauthor{\bsnm{Addario-Berry},~\bfnm{Louigi}\binits{L.}} \AND \bauthor{\bsnm{Reed},~\bfnm{Bruce}\binits{B.}}
(\byear{2009}).
\btitle{Minima in branching random walks}.
\bjournal{Ann. Probab.}
\bvolume{37}
\bpages{1044--1079}.
\bdoi{10.1214/08-AOP428}
\bmrnumber{2537549}
\end{barticle}
\endbibitem

\bibitem{AdTa-07}
\begin{bbook}[author]
\bauthor{\bsnm{Adler},~\bfnm{Robert~J.}\binits{R.~J.}} \AND \bauthor{\bsnm{Taylor},~\bfnm{Jonathan~E.}\binits{J.~E.}}
(\byear{2007}).
\btitle{Random fields and geometry}.
\bseries{Springer Monographs in Mathematics}.
\bpublisher{Springer}, \baddress{New York}.
\bmrnumber{2319516}
\end{bbook}
\endbibitem

\bibitem{Ai-13}
\begin{barticle}[author]
\bauthor{\bsnm{A{\"{\i}}d{\'e}kon},~\bfnm{Elie}\binits{E.}}
(\byear{2013}).
\btitle{Convergence in law of the minimum of a branching random walk}.
\bjournal{Ann. Probab.}
\bvolume{41}
\bpages{1362--1426}.
\end{barticle}
\endbibitem

\bibitem{Angenent1988zeroset}
\begin{barticle}[author]
\bauthor{\bsnm{Angenent},~\bfnm{Sigurd}\binits{S.}}
(\byear{1988}).
\btitle{The zero set of a solution of a parabolic equation}.
\bjournal{J. Reine Angew. Math.}
\bvolume{390}
\bpages{79--96}.
\bdoi{10.1515/crll.1988.390.79}
\bmrnumber{953678}
\end{barticle}
\endbibitem

\bibitem{angenent1991nodalproperties}
\begin{barticle}[author]
\bauthor{\bsnm{Angenent},~\bfnm{Sigurd}\binits{S.}}
(\byear{1991}).
\btitle{Nodal properties of solutions of parabolic equations}.
\bjournal{Rocky Mountain J. Math.}
\bvolume{21}
\bpages{585--592}.
\bdoi{10.1216/rmjm/1181072953}
\bmrnumber{1121527}
\end{barticle}
\endbibitem

\bibitem{Bass98}
\begin{bbook}[author]
\bauthor{\bsnm{Bass},~\bfnm{Richard~F.}\binits{R.~F.}}
(\byear{1998}).
\btitle{Diffusions and elliptic operators}.
\bseries{Probability and its Applications (New York)}.
\bpublisher{Springer-Verlag, New York}.
\bmrnumber{1483890}
\end{bbook}
\endbibitem

\bibitem{BeBrHaHaRo-12}
\begin{barticle}[author]
\bauthor{\bsnm{Berestycki},~\bfnm{Julien}\binits{J.}}, \bauthor{\bsnm{Brunet},~\bfnm{\'Eric}\binits{E.}}, \bauthor{\bsnm{Harris},~\bfnm{John~W.}\binits{J.~W.}}, \bauthor{\bsnm{Harris},~\bfnm{Simon~C.}\binits{S.~C.}} \AND \bauthor{\bsnm{Roberts},~\bfnm{Matthew~I.}\binits{M.~I.}}
(\byear{2015}).
\btitle{Growth rates of the population in a branching {B}rownian motion with an inhomogeneous breeding potential}.
\bjournal{Stochastic Process. Appl.}
\bvolume{125}
\bpages{2096--2145}.
\bmrnumber{3315624}
\end{barticle}
\endbibitem

\bibitem{Bhattacharya1986nonrmal}
\begin{bbook}[author]
\bauthor{\bsnm{Bhattacharya},~\bfnm{Rabi~N.}\binits{R.~N.}} \AND \bauthor{\bsnm{Rao},~\bfnm{R.~Ranga}\binits{R.~R.}}
(\byear{2010}).
\btitle{Normal approximation and asymptotic expansions}.
\bseries{Classics in Applied Mathematics}
\bvolume{64}.
\bpublisher{Society for Industrial and Applied Mathematics (SIAM), Philadelphia, PA}.
\bdoi{10.1137/1.9780898719895.ch1}
\bmrnumber{3396213}
\end{bbook}
\endbibitem

\bibitem{Bolthausen2011recursions}
\begin{barticle}[author]
\bauthor{\bsnm{Bolthausen},~\bfnm{Erwin}\binits{E.}}, \bauthor{\bsnm{Deuschel},~\bfnm{Jean~Dominique}\binits{J.~D.}} \AND \bauthor{\bsnm{Zeitouni},~\bfnm{Ofer}\binits{O.}}
(\byear{2011}).
\btitle{Recursions and tightness for the maximum of the discrete, two dimensional {G}aussian free field}.
\bjournal{Electron. Commun. Probab.}
\bvolume{16}
\bpages{114--119}.
\bdoi{10.1214/ECP.v16-1610}
\bmrnumber{2772390}
\end{barticle}
\endbibitem

\bibitem{Borodin1996HandbookOB}
\begin{bbook}[author]
\bauthor{\bsnm{Borodin},~\bfnm{Andrei~N.}\binits{A.~N.}} \AND \bauthor{\bsnm{Salminen},~\bfnm{Paavo}\binits{P.}}
(\byear{2002}).
\btitle{Handbook of {B}rownian motion---facts and formulae},
\bedition{second} ed.
\bseries{Probability and its Applications}.
\bpublisher{Birkh\"{a}user Verlag, Basel}.
\bdoi{10.1007/978-3-0348-8163-0}
\bmrnumber{1912205}
\end{bbook}
\endbibitem

\bibitem{BoHa-14}
\begin{barticle}[author]
\bauthor{\bsnm{Bovier},~\bfnm{Anton}\binits{A.}} \AND \bauthor{\bsnm{Hartung},~\bfnm{Lisa}\binits{L.}}
(\byear{2014}).
\btitle{The extremal process of two-speed branching {B}rownian motion}.
\bjournal{Electron. J. Probab.}
\bvolume{19}
\bpages{no. 18, 28}.
\bdoi{10.1214/EJP.v19-2982}
\bmrnumber{3164771}
\end{barticle}
\endbibitem

\bibitem{BoHa-15}
\begin{barticle}[author]
\bauthor{\bsnm{Bovier},~\bfnm{Anton}\binits{A.}} \AND \bauthor{\bsnm{Hartung},~\bfnm{Lisa}\binits{L.}}
(\byear{2015}).
\btitle{Variable speed branching {B}rownian motion 1. {E}xtremal processes in the weak correlation regime}.
\bjournal{ALEA Lat. Am. J. Probab. Math. Stat.}
\bvolume{12}
\bpages{261--291}.
\bmrnumber{3351476}
\end{barticle}
\endbibitem

\bibitem{Bramson1983FKPP}
\begin{barticle}[author]
\bauthor{\bsnm{Bramson},~\bfnm{Maury}\binits{M.}}
(\byear{1983}).
\btitle{Convergence of solutions of the {K}olmogorov equation to travelling waves}.
\bjournal{Mem. Amer. Math. Soc.}
\bvolume{44}
\bpages{iv+190}.
\bdoi{10.1090/memo/0285}
\bmrnumber{705746}
\end{barticle}
\endbibitem

\bibitem{Bramson2007tightnessBRW}
\begin{barticle}[author]
\bauthor{\bsnm{Bramson},~\bfnm{Maury}\binits{M.}} \AND \bauthor{\bsnm{Zeitouni},~\bfnm{Ofer}\binits{O.}}
(\byear{2007}).
\btitle{Tightness for the minimal displacement of branching random walk}.
\bjournal{J. Stat. Mech. Theory Exp.}
\bvolume{7}
\bpages{P07010, 12}.
\bdoi{10.1088/1742-5468/2007/07/p07010}
\bmrnumber{2335694}
\end{barticle}
\endbibitem

\bibitem{BramsonZeitouni2009recursion}
\begin{barticle}[author]
\bauthor{\bsnm{Bramson},~\bfnm{Maury}\binits{M.}} \AND \bauthor{\bsnm{Zeitouni},~\bfnm{Ofer}\binits{O.}}
(\byear{2009}).
\btitle{Tightness for a family of recursion equations}.
\bjournal{Ann. Probab.}
\bvolume{37}
\bpages{615--653}.
\bdoi{10.1214/08-AOP414}
\bmrnumber{2510018}
\end{barticle}
\endbibitem

\bibitem{Bramson1978maximaldisplacement}
\begin{barticle}[author]
\bauthor{\bsnm{Bramson},~\bfnm{Maury~D.}\binits{M.~D.}}
(\byear{1978}).
\btitle{Maximal displacement of branching {B}rownian motion}.
\bjournal{Comm. Pure Appl. Math.}
\bvolume{31}
\bpages{531--581}.
\bdoi{10.1002/cpa.3160310502}
\bmrnumber{494541}
\end{barticle}
\endbibitem

\bibitem{CD2020}
\begin{barticle}[author]
\bauthor{\bsnm{\cerny},~\bfnm{Ji\v{r}\'{\i}}\binits{J.}} \AND \bauthor{\bsnm{Drewitz},~\bfnm{Alexander}\binits{A.}}
(\byear{2020}).
\btitle{Quenched invariance principles for the maximal particle in branching random walk in random environment and the parabolic {A}nderson model}.
\bjournal{Ann. Probab.}
\bvolume{48}
\bpages{94--146}.
\bdoi{10.1214/19-AOP1347}
\bmrnumber{4079432}
\end{barticle}
\endbibitem

\bibitem{CDS2021}
\begin{barticle}[author]
\bauthor{\bsnm{\cerny},~\bfnm{Ji\v{r}\'{\i}}\binits{J.}}, \bauthor{\bsnm{Drewitz},~\bfnm{Alexander}\binits{A.}} \AND \bauthor{\bsnm{Schmitz},~\bfnm{Lars}\binits{L.}}
(\byear{2023}).
\btitle{({U}n-)bounded transition fronts for the parabolic {A}nderson model and the randomized {F}-{KPP} equation}.
\bjournal{Ann. Appl. Probab.}
\bvolume{33}
\bpages{2342--2373}.
\bdoi{10.1214/22-aap1869}
\bmrnumber{4583673}
\end{barticle}
\endbibitem

\bibitem{Dembo2021limitlaw}
\begin{barticle}[author]
\bauthor{\bsnm{Dembo},~\bfnm{Amir}\binits{A.}}, \bauthor{\bsnm{Rosen},~\bfnm{Jay}\binits{J.}} \AND \bauthor{\bsnm{Zeitouni},~\bfnm{Ofer}\binits{O.}}
(\byear{2021}).
\btitle{Limit law for the cover time of a random walk on a binary tree}.
\bjournal{Ann. Inst. Henri Poincar\'{e} Probab. Stat.}
\bvolume{57}
\bpages{830--855}.
\bdoi{10.1214/20-aihp1098}
\bmrnumber{4260486}
\end{barticle}
\endbibitem

\bibitem{DS2022}
\begin{barticle}[author]
\bauthor{\bsnm{Drewitz},~\bfnm{Alexander}\binits{A.}} \AND \bauthor{\bsnm{Schmitz},~\bfnm{Lars}\binits{L.}}
(\byear{2022}).
\btitle{Invariance {P}rinciples and {L}og-{D}istance of {F}-{KPP} {F}ronts in a {R}andom {M}edium}.
\bjournal{Arch. Ration. Mech. Anal.}
\bvolume{246}
\bpages{877--955}.
\bdoi{10.1007/s00205-022-01824-x}
\bmrnumber{4514066}
\end{barticle}
\endbibitem

\bibitem{Ducrot2014propagatingterrace}
\begin{barticle}[author]
\bauthor{\bsnm{Ducrot},~\bfnm{Arnaud}\binits{A.}}, \bauthor{\bsnm{Giletti},~\bfnm{Thomas}\binits{T.}} \AND \bauthor{\bsnm{Matano},~\bfnm{Hiroshi}\binits{H.}}
(\byear{2014}).
\btitle{Existence and convergence to a propagating terrace in one-dimensional reaction-diffusion equations}.
\bjournal{Trans. Amer. Math. Soc.}
\bvolume{366}
\bpages{5541--5566}.
\bdoi{10.1090/S0002-9947-2014-06105-9}
\bmrnumber{3240934}
\end{barticle}
\endbibitem

\bibitem{evanswilliams1999zerocrossings}
\begin{barticle}[author]
\bauthor{\bsnm{Evans},~\bfnm{Steven~N.}\binits{S.~N.}} \AND \bauthor{\bsnm{Williams},~\bfnm{Ruth~J.}\binits{R.~J.}}
(\byear{1999}).
\btitle{Transition operators of diffusions reduce zero-crossing}.
\bjournal{Trans. Amer. Math. Soc.}
\bvolume{351}
\bpages{1377--1389}.
\bdoi{10.1090/S0002-9947-99-02341-7}
\bmrnumber{1615955}
\end{barticle}
\endbibitem

\bibitem{FangZeitouni2012BRWRE}
\begin{barticle}[author]
\bauthor{\bsnm{Fang},~\bfnm{Ming}\binits{M.}} \AND \bauthor{\bsnm{Zeitouni},~\bfnm{Ofer}\binits{O.}}
(\byear{2012}).
\btitle{Branching random walks in time inhomogeneous environments}.
\bjournal{Electron. J. Probab.}
\bvolume{17}
\bpages{no. 67, 18}.
\bdoi{10.1214/EJP.v17-2253}
\bmrnumber{2968674}
\end{barticle}
\endbibitem

\bibitem{FaZe-12b}
\begin{barticle}[author]
\bauthor{\bsnm{Fang},~\bfnm{Ming}\binits{M.}} \AND \bauthor{\bsnm{Zeitouni},~\bfnm{Ofer}\binits{O.}}
(\byear{2012}).
\btitle{Slowdown for time inhomogeneous branching {B}rownian motion}.
\bjournal{J. Stat. Phys.}
\bvolume{149}
\bpages{1--9}.
\bdoi{10.1007/s10955-012-0581-z}
\bmrnumber{2981635}
\end{barticle}
\endbibitem

\bibitem{Freidlin1985functionalintegration}
\begin{bbook}[author]
\bauthor{\bsnm{Freidlin},~\bfnm{Mark}\binits{M.}}
(\byear{1985}).
\btitle{Functional integration and partial differential equations}.
\bseries{Annals of Mathematics Studies}
\bvolume{109}.
\bpublisher{Princeton University Press, Princeton, NJ}.
\bdoi{10.1515/9781400881598}
\bmrnumber{833742}
\end{bbook}
\endbibitem

\bibitem{Galaktionov2004sturmian}
\begin{bbook}[author]
\bauthor{\bsnm{Galaktionov},~\bfnm{Victor~A.}\binits{V.~A.}}
(\byear{2004}).
\btitle{Geometric {S}turmian theory of nonlinear parabolic equations and applications}
\bvolume{3}.
\bpublisher{Chapman \& Hall/CRC, Boca Raton, FL}.
\bdoi{10.1201/9780203998069}
\bmrnumber{2059317}
\end{bbook}
\endbibitem

\bibitem{GT2001}
\begin{bbook}[author]
\bauthor{\bsnm{Gilbarg},~\bfnm{David}\binits{D.}} \AND \bauthor{\bsnm{Trudinger},~\bfnm{Neil~S.}\binits{N.~S.}}
(\byear{2001}).
\btitle{Elliptic partial differential equations of second order}.
\bseries{Classics in Mathematics}.
\bpublisher{Springer-Verlag, Berlin}.
\bmrnumber{1814364}
\end{bbook}
\endbibitem

\bibitem{Hamel2016periodic}
\begin{barticle}[author]
\bauthor{\bsnm{Hamel},~\bfnm{Fran\c{c}ois}\binits{F.}}, \bauthor{\bsnm{Nolen},~\bfnm{James}\binits{J.}}, \bauthor{\bsnm{Roquejoffre},~\bfnm{Jean-Michel}\binits{J.-M.}} \AND \bauthor{\bsnm{Ryzhik},~\bfnm{Lenya}\binits{L.}}
(\byear{2016}).
\btitle{The logarithmic delay of {KPP} fronts in a periodic medium}.
\bjournal{J. Eur. Math. Soc. (JEMS)}
\bvolume{18}
\bpages{465--505}.
\bdoi{10.4171/JEMS/595}
\bmrnumber{3463416}
\end{barticle}
\endbibitem

\bibitem{HoReSo-22}
\begin{barticle}[author]
\bauthor{\bsnm{Hou},~\bfnm{Haojie}\binits{H.}}, \bauthor{\bsnm{Ren},~\bfnm{Yan-Xia}\binits{Y.-X.}} \AND \bauthor{\bsnm{Song},~\bfnm{Renming}\binits{R.}}
(\byear{2023}).
\btitle{Invariance principle for the maximal position process of branching {B}rownian motion in random environment}.
\bjournal{Electron. J. Probab.}
\bvolume{28}
\bpages{Paper No. 65, 63}.
\bdoi{10.1214/23-ejp956}
\bmrnumber{4585411}
\end{barticle}
\endbibitem

\bibitem{HuShi2009minimal}
\begin{barticle}[author]
\bauthor{\bsnm{Hu},~\bfnm{Yueyun}\binits{Y.}} \AND \bauthor{\bsnm{Shi},~\bfnm{Zhan}\binits{Z.}}
(\byear{2009}).
\btitle{Minimal position and critical martingale convergence in branching random walks, and directed polymers on disordered trees}.
\bjournal{Ann. Probab.}
\bvolume{37}
\bpages{742--789}.
\bdoi{10.1214/08-AOP419}
\bmrnumber{2510023}
\end{barticle}
\endbibitem

\bibitem{Ikeda1968branching}
\begin{barticle}[author]
\bauthor{\bsnm{Ikeda},~\bfnm{Nobuyuki}\binits{N.}}, \bauthor{\bsnm{Nagasawa},~\bfnm{Masao}\binits{M.}} \AND \bauthor{\bsnm{Watanabe},~\bfnm{Shinzo}\binits{S.}}
(\byear{1968}).
\btitle{Branching {M}arkov processes. {I}}.
\bjournal{J. Math. Kyoto Univ.}
\bvolume{8}
\bpages{233--278}.
\bdoi{10.1215/kjm/1250524137}
\bmrnumber{232439}
\end{barticle}
\endbibitem

\bibitem{KPP1937etude}
\begin{bmisc}[author]
\bauthor{\bsnm{Kolmogorov},~\bfnm{Andrey}\binits{A.}}, \bauthor{\bsnm{Petrovsky},~\bfnm{Ivan}\binits{I.}} \AND \bauthor{\bsnm{Piskunov},~\bfnm{Nikolai}\binits{N.}}
(\byear{1937}).
\btitle{{\'Etude de l'\'equation de la diffusion avec croissance de la quantite de mati\`ere et son application \`a un probl\`eme biologique}}.
\bhowpublished{{Bull. Univ. \'Etat Moscou, S\'er. Int., Sect. A: Math. et M\'ecan. 1, Fasc. 6, 1-25 (1937)}}.
\end{bmisc}
\endbibitem

\bibitem{Koenig2016PAM}
\begin{bbook}[author]
\bauthor{\bsnm{K{\"o}nig},~\bfnm{Wolfgang}\binits{W.}}
(\byear{2016}).
\btitle{The parabolic {A}nderson model}.
\bseries{Pathways in Mathematics}.
\bpublisher{{Birk\-h\"{a}user}}.
\bdoi{10.1007/978-3-319-33596-4}
\bmrnumber{3526112}
\end{bbook}
\endbibitem

\bibitem{Kriechbaum2021subsequential}
\begin{barticle}[author]
\bauthor{\bsnm{Kriechbaum},~\bfnm{Xaver}\binits{X.}}
(\byear{2021}).
\btitle{Subsequential tightness for branching random walk in random environment}.
\bjournal{Electron. Commun. Probab.}
\bvolume{26}
\bpages{Paper No. 16, 12}.
\bmrnumber{4236686}
\end{barticle}
\endbibitem

\bibitem{Kriechbaum2022BRWRE}
\begin{bmisc}[author]
\bauthor{\bsnm{Kriechbaum},~\bfnm{Xaver}\binits{X.}}
(\byear{2022}).
\btitle{Tightness for branching random walk in time-inhomogeneous random environment}.
\bdoi{10.48550/ARXIV.2205.11643}
\end{bmisc}
\endbibitem

\bibitem{LSU68}
\begin{bbook}[author]
\bauthor{\bsnm{Lady\v{z}enskaja},~\bfnm{O.~A.}\binits{O.~A.}}, \bauthor{\bsnm{Solonnikov},~\bfnm{V.~A.}\binits{V.~A.}} \AND \bauthor{\bsnm{Ural'ceva},~\bfnm{N.~N.}\binits{N.~N.}}
(\byear{1968}).
\btitle{Linear and quasi-linear equations of parabolic type}.
\bseries{Translations of Mathematical Monographs, Vol. 23}.
\bpublisher{American Mathematical Society, Providence, R.I.}
\bmrnumber{0241822}
\end{bbook}
\endbibitem

\bibitem{LaSe-88}
\begin{barticle}[author]
\bauthor{\bsnm{Lalley},~\bfnm{S.}\binits{S.}} \AND \bauthor{\bsnm{Sellke},~\bfnm{T.}\binits{T.}}
(\byear{1988}).
\btitle{Traveling waves in inhomogeneous branching {B}rownian motions. {I}}.
\bjournal{Ann. Probab.}
\bvolume{16}
\bpages{1051--1062}.
\bmrnumber{942755}
\end{barticle}
\endbibitem

\bibitem{LaSe-89}
\begin{barticle}[author]
\bauthor{\bsnm{Lalley},~\bfnm{S.}\binits{S.}} \AND \bauthor{\bsnm{Sellke},~\bfnm{T.}\binits{T.}}
(\byear{1989}).
\btitle{Travelling waves in inhomogeneous branching {B}rownian motions. {II}}.
\bjournal{Ann. Probab.}
\bvolume{17}
\bpages{116--127}.
\bmrnumber{972775}
\end{barticle}
\endbibitem

\bibitem{Lubetzky2018MaximumOB}
\begin{barticle}[author]
\bauthor{\bsnm{Lubetzky},~\bfnm{Eyal}\binits{E.}}, \bauthor{\bsnm{Thornett},~\bfnm{Chris}\binits{C.}} \AND \bauthor{\bsnm{Zeitouni},~\bfnm{Ofer}\binits{O.}}
(\byear{2022}).
\btitle{Maximum of branching {B}rownian motion in a periodic environment}.
\bjournal{Ann. Inst. Henri Poincar\'{e} Probab. Stat.}
\bvolume{58}
\bpages{2065--2093}.
\bdoi{10.1214/21-aihp1219}
\bmrnumber{4492971}
\end{barticle}
\endbibitem

\bibitem{MaZe-13}
\begin{barticle}[author]
\bauthor{\bsnm{Maillard},~\bfnm{Pascal}\binits{P.}} \AND \bauthor{\bsnm{Zeitouni},~\bfnm{Ofer}\binits{O.}}
(\byear{2016}).
\btitle{Slowdown in branching {B}rownian motion with inhomogeneous variance}.
\bjournal{Ann. Inst. Henri Poincar\'e Probab. Stat.}
\bvolume{52}
\bpages{1144--1160}.
\bmrnumber{3531703}
\end{barticle}
\endbibitem

\bibitem{Ma-13}
\begin{barticle}[author]
\bauthor{\bsnm{Mallein},~\bfnm{Bastien}\binits{B.}}
(\byear{2015}).
\btitle{Maximal displacement of a branching random walk in time-inhomogeneous environment}.
\bjournal{Stochastic Process. Appl.}
\bvolume{125}
\bpages{3958--4019}.
\bmrnumber{3373310}
\end{barticle}
\endbibitem

\bibitem{mckean1975KPP}
\begin{barticle}[author]
\bauthor{\bsnm{McKean},~\bfnm{H.~P.}\binits{H.~P.}}
(\byear{1975}).
\btitle{Application of {B}rownian motion to the equation of {K}olmogorov-{P}etrovskii-{P}iskunov}.
\bjournal{Comm. Pure Appl. Math.}
\bvolume{28}
\bpages{323--331}.
\bdoi{10.1002/cpa.3160280302}
\bmrnumber{400428}
\end{barticle}
\endbibitem

\bibitem{Nadin2012criticalwaves}
\begin{barticle}[author]
\bauthor{\bsnm{Nadin},~\bfnm{Gr\'{e}goire}\binits{G.}}
(\byear{2015}).
\btitle{Critical travelling waves for general heterogeneous one-dimensional reaction-diffusion equations}.
\bjournal{Ann. Inst. H. Poincar\'{e} C Anal. Non Lin\'{e}aire}
\bvolume{32}
\bpages{841--873}.
\bdoi{10.1016/j.anihpc.2014.03.007}
\bmrnumber{3390087}
\end{barticle}
\endbibitem

\bibitem{Neuman2021maximal}
\begin{barticle}[author]
\bauthor{\bsnm{Neuman},~\bfnm{Eyal}\binits{E.}} \AND \bauthor{\bsnm{Zheng},~\bfnm{Xinghua}\binits{X.}}
(\byear{2021}).
\btitle{On the maximal displacement of near-critical branching random walks}.
\bjournal{Probab. Theory Related Fields}
\bvolume{180}
\bpages{199--232}.
\bdoi{10.1007/s00440-021-01042-8}
\bmrnumber{4265021}
\end{barticle}
\endbibitem

\bibitem{RogrsWilliams2000Markov}
\begin{bbook}[author]
\bauthor{\bsnm{Rogers},~\bfnm{L.~C.~G.}\binits{L.~C.~G.}} \AND \bauthor{\bsnm{Williams},~\bfnm{David}\binits{D.}}
(\byear{2000}).
\btitle{Diffusions, {M}arkov processes, and martingales. {V}ol. 2}.
\bpublisher{Cambridge University Press, Cambridge}.
\bdoi{10.1017/CBO9781107590120}
\bmrnumber{1780932}
\end{bbook}
\endbibitem

\bibitem{Skorohod1964branching}
\begin{barticle}[author]
\bauthor{\bsnm{Skorohod},~\bfnm{A.~V.}\binits{A.~V.}}
(\byear{1964}).
\btitle{Branching diffusion processes}.
\bjournal{Teor. Verojatnost. i Primenen.}
\bvolume{9}
\bpages{492--497}.
\bmrnumber{0168030}
\end{barticle}
\endbibitem

\bibitem{Sturm1836memoires}
\begin{barticle}[author]
\bauthor{\bsnm{Sturm},~\bfnm{Charles}\binits{C.}}
(\byear{1836}).
\btitle{M{\'{e}}moire sur une classe d'équations {\`{a} } diff{\'{e}}rences partielles}.
\bjournal{J. Math. Pures et Appl.}
\bvolume{1}
\bpages{373--444}.
\end{barticle}
\endbibitem

\end{thebibliography}

\end{document}